\documentclass[12pt,reqno]{elsarticle}

\makeatletter
\def\ps@pprintTitle{%
 \let\@oddhead\@empty
 \let\@evenhead\@empty
 \def\@oddfoot{}%
 \let\@evenfoot\@oddfoot}
\makeatother

\usepackage{amsmath,amsthm,amscd,wrapfig,amsfonts,amssymb}
\usepackage{enumitem,comment,mathabx}
\usepackage{booktabs,multirow}

\usepackage{url,parskip,caption}
\usepackage[top=22mm, bottom=22mm, left=22mm, right=22mm]{geometry}

\newtheorem{corollary}{Corollary}
\newtheorem{theorem}{Theorem}
\newtheorem{lemma}{Lemma}
\newtheorem{remark}{Remark}
\newtheorem{definition}{Defintion}

\usepackage{palatino}

\newdimen\CdotAxis
\newcommand*{\CdotAux}[3]{%
  {%
    \settoheight\CdotAxis{$#2\vcenter{}$}%
    \sbox0{%
      \raisebox\CdotAxis{%
        \scalebox{#1}{%
          \raisebox{-\CdotAxis}{%
            $\mathsurround=0pt #2#3$%
          }%
        }%
      }%
    }%
    \dp0=0pt %
    \sbox2{$#2\bullet$}%
    \ifdim\ht2<\ht0 %
      \ht0=\ht2 %
    \fi
    \sbox2{$\mathsurround=0pt #2#3$}%
    \hbox to \wd2{\hss\usebox{0}\hss}%
  }%
}

\DeclareMathOperator{\sign}{sign}
\DeclareMathOperator{\erf}{erf}

\providecommand{\Czero}{\ensuremath{C_0\left(\mathbb{R}\right)}}

\providecommand{\Lp}[1]{\ensuremath{L^{#1}\left(\mathbb{R}\right)}}
\providecommand{\C}[1]{\ensuremath{C^{#1}\left(\mathbb{R}\right)}}

\providecommand{\OpT}[1]{\ensuremath{T\left[\hskip .1em{#1}\hskip .1em \right]}}

\providecommand{\OpWb}[1]{\ensuremath{W_b\left[\hskip .1em{#1}\hskip .1em \right]}}
\providecommand{\OpWTb}[1]{\ensuremath{\widetilde{W}_b\left[\hskip .1em{#1}\hskip .1em\right]}}

\providecommand{\OpR}[1]{\ensuremath{R\left[\hskip .1em{#1}\hskip .1em \right]}}

\providecommand{\OpTb}[1]{\ensuremath{T_b\left[\hskip .1em{#1}\hskip .1em \right]}}

\providecommand{\OpS}[1]{\ensuremath{S\left[{\hskip .1em {#1}\hskip .1em} \right]}}

\providecommand{\Sr}{S(\mathbb{R})}

\providecommand{\e}[1]{\ensuremath{\times 10^{#1}}}

\makeatletter
\g@addto@macro\normalsize{%
  \setlength\abovedisplayskip{.4em}
  \setlength\belowdisplayskip{.4em}
  \setlength\abovedisplayshortskip{.4em}
  \setlength\belowdisplayshortskip{.4em}
}

\begin{document}

\begin{frontmatter}
\begin{abstract}
We observe that solutions of a large class of highly oscillatory second order linear 
ordinary differential equations can be approximated using nonoscillatory 
phase functions.  In addition, we describe numerical experiments which 
illustrate important implications of this fact.  For example,
that many special functions of great interest --- such as the Bessel functions
$J_\nu$ and $Y_\nu$ ---  can be evaluated accurately 
using a number of operations which is $O(1)$ in the order $\nu$.
The present paper is devoted to the development of an  analytical  apparatus.  
Numerical aspects of this work will be reported at a later date.

\end{abstract}


\title
{
On the existence of nonoscillatory phase functions 
for second order ordinary differential equations 
in the high-frequency regime
}

\author[vr]{Zhu Heitman}
\author[jb]{James Bremer\corref{cor1}}
\ead{bremer@math.ucdavis.edu}
\author[vr]{Vladimir Rokhlin}

\cortext[cor1]{Corresponding author}

\address[vr]{Department of Computer Science, Yale University}
\address[jb]{Department of Mathematics, University of California, Davis}

\end{frontmatter}




\begin{section}{Introduction} 

Given a differential equation
\begin{equation}
y''(t) + \lambda^2 q(t) y(t)  = 0\ \ \ \mbox{for all}\ \  0\leq t \leq 1,
\label{introduction:original_equation}
\end{equation}
where $\lambda$ is a real number and $q:[0,1]\to\mathbb{R}$ is smooth and strictly positive,
a sufficiently smooth $\alpha:[0,1]\to\mathbb{R}$ is a phase function
for (\ref{introduction:original_equation})
if the pair of functions $u,v$ defined by the formulas
\begin{equation}
u(t) = \frac{\cos(\alpha(t)) }{\left|\alpha'(t)\right|^{1/2}}
\label{introduction:u}
\end{equation}
and
\begin{equation}
v(t) = \frac{\sin(\alpha(t)) }{\left|\alpha'(t)\right|^{1/2}}
\label{introduction:v}
\end{equation}
form a basis in the space of solutions of (\ref{introduction:original_equation}).
Phase functions have been extensively studied: they were first introduced in \cite{Kummer},
play a key role in the  theory of global transformations
of ordinary differential equations  \cite{Boruvka,Neuman}, and are
an important element in the theory of special functions
\cite{Spigler-Vianello,Goldstein-Thaler,NISTHandbook,Andrews-Askey-Roy}.


Despite this long history, an important property of phase functions appears to have been
overlooked.  Specifically, that when the function $q$ is nonoscillatory, solutions
of the equation (\ref{introduction:original_equation}) can be accurately represented
using a nonoscillatory phase function.  


This is somewhat surprising since $\alpha$
is a phase function for (\ref{introduction:original_equation}) if and only
if it satisfies the third order nonlinear ordinary differential equation
\begin{equation}
\left(\alpha'(t)\right)^2 = \lambda^2q(t) - \frac{1}{2}\frac{\alpha'''(t)}{\alpha'(t)}
+ \frac{3}{4} \left(\frac{\alpha''(t)}{\alpha'(t)}\right)^2
\ \ \ \mbox{for all}\ \ \ 0\leq t \leq 1.
\label{introduction:kummers_equation}
\end{equation}
The equation (\ref{introduction:kummers_equation}) was introduced
in \cite{Kummer}, and  and we will refer to it as Kummer's equation.
The form of (\ref{introduction:kummers_equation}) and the appearance
of $\lambda$ in it
suggests 
that its solutions will be oscillatory  --- and most of them are.
However, Bessel's equation 
\begin{equation}
y''(t) + \left(1-\frac{\lambda^2-1/4}{t^2} \right) y(t) = 0
\ \ \ \mbox{for all}\ \ 0 < t < \infty
\label{introduction:bessel}
\end{equation}
furnishes a nontrivial example of an equation which admits a nonoscillatory phase function
regardless of the value of $\lambda$.   
If we define $u,v$ by the formulas
\begin{equation}
u(t) = \sqrt{\frac{\pi t}{2}} J_\lambda(t)
\end{equation}
and
\begin{equation}
v(t) = \sqrt{\frac{\pi t}{2}} Y_\lambda(t),
\end{equation}
where $J_\lambda$ and $Y_\lambda$ denote  the Bessel functions of the first and second
kinds of order $\lambda$, and let $\alpha$ be defined by the relations
(\ref{introduction:u}),(\ref{introduction:v}), then
\begin{equation}
\alpha'(t) =  \frac { 2 }{\pi t}\frac{1}{J_\lambda^2(t) + Y_\lambda^2(t)}.
\label{introduction:bessel_prime}
\end{equation}
It can be easily verified that 
(\ref{introduction:bessel_prime}) is a nonoscillatory.
The existence of this nonoscillatory
phase function for Bessel's equation is the basis of several
methods for the evaluation of Bessel functions of large orders
and for the computation of their zeros
\cite{Goldstein-Thaler,Rokhlin-Bremer-Others,Spigler-Vianello2}.  

The general situation is not quite so favorable: there  need not 
exist a nonoscillatory function $\alpha$ such that
(\ref{introduction:u}) and (\ref{introduction:v}) are exact
solutions of (\ref{introduction:original_equation}).
However,  assuming that $q$ is nonoscillatory and $\lambda$ is sufficiently large,
there exists a nonoscillatory function $\alpha$ such that
(\ref{introduction:u}), (\ref{introduction:v}) approximate
solutions of (\ref{introduction:original_equation}) with spectral
accuracy (i.e., the approximation errors decay exponentially with $\lambda$).



%


To see that this claim is plausible,  
we apply Newton's method for the solution of nonlinear equations to Kummer's 
equation (\ref{introduction:kummers_equation}).
In doing so, it will  be convenient to move the setting of our analysis from the interval 
$[0,1]$ to the real line so that we can use the Fourier transform
to quantity the notion of ``nonoscillatory.''
Suppose that the extension of $q$ to the real line 
is smooth and strictly positive, and such that $\log(q)$ is a smooth
function with rapidly decaying  Fourier transform.
Letting 
\begin{equation}
\left(\alpha'(t)\right)^2 = \lambda^2 \exp(r(t))
\end{equation}
 in (\ref{introduction:kummers_equation}) 
yields the logarithm form of Kummer's equation:
\begin{equation}
r''(t) - \frac{1}{4}\left(r'(t)\right)^2 + 4 \lambda^2\left( \exp(r(t)) - q(t)\right) = 0
\ \ \ \mbox{for all}\ \ t \in \mathbb{R}.
\label{introduction:kummer_logarithm_form}
\end{equation}
We use  $\{r_n\}$ to denote the sequence of Newton iterates
for the equation (\ref{introduction:kummer_logarithm_form}) obtained from the initial guess
\begin{equation}
r_0(t) = \log(q(t)).
\end{equation}
The function $r_0$ corresponds to  the first order 
WKB approximations for (\ref{introduction:original_equation}).
That is to say that if we insert the associated phase function
\begin{equation}
\alpha_0(t) = \lambda  \int_0^t \exp\left(\frac{1}{2}r_0(u)\right)\ du
= \lambda \int_0^t \sqrt{q(u)} du
\end{equation}
into (\ref{introduction:u}),(\ref{introduction:v}), then 
\begin{equation}
u(t) = 
q^{-1/4}(t)\cos\left(\lambda \int_0^t \sqrt{q(u)}\ du\right)
\end{equation}
and
\begin{equation}
v(t) = 
q^{-1/4}(t)\sin\left(\lambda \int_0^t \sqrt{q(u)}\ du\right).
\end{equation}
%
For each $n \geq 0$, $r_{n+1}$ is obtained from $r_n$ 
by solving the linearized equation
\begin{equation}
 h''(t) - \frac{1}{2}r_n'(t)h'(t) + 
4\lambda^2\exp\left(r_n(t)\right)h(t) = f_n(t)
\ \ \ \mbox{for all} \ \ t \in \mathbb{R}, 
\label{introduction:newton_linearization}
\end{equation}
where
\begin{equation}
f_n(t) = 
-r_n''(t) + \frac{1}{4}\left(r_n'(t)\right)^2 - 4\lambda^2 \left(\exp\left(r_n(t)\right)-q(t)\right),
\end{equation}
and letting
\begin{equation}
r_{n+1}(t) = r_n(t)+h(t).
\end{equation}
By introducing the change of variables
\begin{equation}
x(t) = \int_0^t \exp\left(\frac{r_n(u)}{2}\right)\ du
\label{introduction:change_of_variables}
\end{equation}
into (\ref{introduction:newton_linearization}), we transform
it  into the inhomogeneous Helmholtz equation
\begin{equation}
h''(x) + 4 \lambda^2 h(x) =
g_n(x)
 \ \ \ \mbox{for all}\ \ x \in \mathbb{R},
\label{introduction:helmholtz}
\end{equation}
where
\begin{equation}
g_n(x) =  
\exp\left(-r_n(x)\right)f_n(x).
\end{equation}
Suppose that $\widehat{g_n}$ decays rapidly (when $n=0$, this is a
consequence of  our assumption that $\log(q)$ has a rapidly decaying Fourier transform)
and let $h^*$ be the solution 
of (\ref{introduction:helmholtz}) whose Fourier transform is 
\begin{equation}
\widehat{h^*}(\xi) = \frac{\widehat{g_n}(\xi)}{4\lambda^2 - \xi^2}.
\label{introduction:ghat1}
\end{equation}
Since $\widehat{h^*}(\xi)$ is singular when $\xi = \pm 2\lambda$,
$h^*$ will necessarily have a component which oscillates at frequency $2\lambda$.
However, according to (\ref{introduction:ghat1}),
the  $\Lp{\infty}$ norm of that component is
\begin{equation}
\frac{\widehat{g_n}(2\lambda)}{4\lambda}.
\label{introduction:libound}
\end{equation}
In fact, by rearranging (\ref{introduction:ghat1}) as 
\begin{equation}
\widehat{h^*}(\xi) 
= \frac{1}{4\lambda} \left( 
\frac{\widehat{g_n}(\xi)}{2\lambda - \xi}
+
\frac{\widehat{g_n}(\xi)}{2\lambda + \xi}
\right)
\label{introduction:ghat2}
\end{equation}
and decomposing each of the terms on the right-hand side  of (\ref{introduction:ghat2}) as 
\begin{equation}
\frac{\widehat{g_n}(\xi)}{2\lambda\pm \xi }=
\frac{1}{4\lambda}
\left(
\frac{\widehat{g_n}(\xi) -  \widehat{g_n}(\mp 2\lambda)
\exp \left( - (2\lambda\pm \xi)^2\right)}
{2\lambda \pm \xi}
+
 \widehat{g_n}(\mp 2\lambda)
\frac{\exp \left( - (2\lambda\pm \xi)^2\right)}
{2\lambda \pm \xi}
\right),
\label{introduction:ghat3}
\end{equation}
we obtain
\begin{equation}
h^*(x) = h_0(x) + h_1(x),
\label{introduction:hstar}
\end{equation}
where $h_0$ is defined by the formula
\begin{equation}
\widehat{h_0}(\xi) 
=\frac{1}{4\lambda}
\left(
\frac{\widehat{g_n}(\xi) -  \widehat{g_n}(- 2\lambda)\exp \left( - (2\lambda+\xi)^2\right)}
{2\lambda + \xi}
+\frac{\widehat{g_n}(\xi) -  \widehat{g_n}(2\lambda)\exp \left( - (2\lambda-\xi)^2\right)}
{2\lambda - \xi}\right),
\label{introduction:h0hat}
\end{equation}
and $h_1$ is defined by the formula
\begin{equation}
\begin{aligned}
\widehat{h_1}(\xi) 
= 
\frac{1}{4\lambda}
\left(
\widehat{g_n}(- 2\lambda)
\frac{  \exp \left( - (2\lambda+\xi)^2\right)}
{2\lambda + \xi}
+
\widehat{g_n}(2\lambda)
\frac{  \exp \left( - (2\lambda-\xi)^2\right)}
{2\lambda - \xi}
\right).
\end{aligned}
\label{introduction:h1hat}
\end{equation}
Since the factor in the denominator in (\ref{introduction:h0hat})
has been canceled and both $\widehat{g_n}$ and the Gaussian function are smooth
and rapidly decaying,
$\widehat{h_0}$ is also smooth and rapidly decaying.
Meanwhile, a straightforward calculation shows that the Fourier transform of 
\begin{equation}
\frac{1}{2i}  \erf\left(\frac{x}{2}\right) 
\exp(2 \lambda i x)
\end{equation}
is
\begin{equation}
\frac{  \exp \left( -(2\lambda-\xi)^2\right)}
{2\lambda - \xi},
\end{equation}
so that (\ref{introduction:h1hat}) implies that
\begin{equation}
\begin{aligned}
h_1(x) 
&= 
\frac{1}{4\lambda}
\left(
\widehat{g_n}(-2\lambda)
\frac{1}{2i}  \erf\left(\frac{x}{2}\right) 
\exp(2 \lambda i x)
-
\widehat{g_n}(2\lambda)
\frac{1}{2i}  \erf\left(\frac{x}{2}\right) 
\exp(-2 \lambda i x)
\right).
\label{introduction:h1_complicated}
\end{aligned}
\end{equation}
Since $g_n$ is real-valued, $\widehat{g_n}(2\lambda) = \widehat{g_n}(-2\lambda)$.
Inserting this into  (\ref{introduction:h1_complicated}) yields
\begin{equation}
\begin{aligned}
h_1(x) =
\frac{\widehat{g}(2\lambda)}{4\lambda}
\erf\left(\frac{x}{2}\right)
\sin\left(2\lambda x \right),
\label{introduction:h1_simple}
\end{aligned}
\end{equation}
which makes it clear that the
$\Lp{\infty}$ norm of $h_1$ is $(4\lambda)^{-1} \widehat{g_n}(2\lambda)$.

In (\ref{introduction:hstar}), the solution of (\ref{introduction:helmholtz}) 
is decomposed as the sum of a nonoscillatory function $h_0$ 
and a highly oscillatory function $h_1$ of small magnitude.
However, the solution of (\ref{introduction:newton_linearization})
is actually given by the function
\begin{equation}
h^*(x(t)) = h_0(x(t)) + h_1(x(t))
\end{equation}
obtained by reversing the change of variables (\ref{introduction:change_of_variables}).  
But since $x(t)$ is nonoscillatory and the Fourier transform of $h_0(x)$ decays rapidly,
we expect that the composition $h_0(x(t))$ will also have a rapidly decaying Fourier transform.
The $\Lp{\infty}$ norm of $h_1(x(t))$ is, of course, the same as that of
$h_1(x)$.
So the solution of the linearized equation (\ref{introduction:newton_linearization})
can be written as the sum of a nonoscillatory function $h_0(x(t))$ and a highly oscillatory
function $h_1(x(t))$ of negligible magnitude.

If, in each iteration of the Newton procedure, we approximate
the solution of (\ref{introduction:newton_linearization}) by constructing
$h^*(x(t))$ and discarding the oscillatory term $h_1(t(x))$
of small magnitude, then it is plausible that 
we will arrive at an approximate  solution $r(t)$ 
of the logarithm form of Kummer's equation which is nonoscillatory,
assuming the Fourier transform of $r_0(t) = \log(q(t))$ decays rapidly enough
and $\lambda$ is sufficiently large.

Most of the remainder of this paper is devoted to developing a rigorous argument 
to replace the preceding heuristic discussion.  
In Section~\ref{section:preliminaries},
we summarize a number of well-known mathematical facts to be used throughout
this article.  
In Section~\ref{section:inteq}, we reformulate Kummer's equation as a nonlinear integral equation.
Once that is accomplished, we are in a position to state the principal result of the paper
and discuss its implications; this is done in Section~\ref{section:overview}.
The proof of this principal result is contained in Sections~\ref{section:bandlimit},
\ref{section:bandlimit:solution}, \ref{section:spectrum} and \ref{section:solution}.
Section~ \ref{section:backwards} contains an elementary proof 
of a relevant error estimate for second order differential equations of the form 
(\ref{introduction:original_equation}).

In Section~\ref{section:numerics}, we present the results of numerical experiments
concerning the evaluation of special functions.  
The details of  our numerical algorithm will be reported at a later date.

We conclude with a few brief remarks in  Section~\ref{section:conclusion}.

\label{section:introduction}
\end{section}

\begin{section}{Preliminaries}

%
%
%
\begin{subsection}{Modified Bessel functions}
The modified Bessel function $K_\nu(t)$ of the first kind of order $\nu$
is defined for $t \in \mathbb{R}$ and $\nu \in \mathbb{C}$
by the formula
\begin{equation}
K_\nu(t) = \int_0^\infty \exp\left(-t\cosh\left(t\right)\right)\cosh(\nu t)\ dt.
\end{equation}
The following bound on the ratio of $K_{\nu+1}$ to $K_\nu$ can be
found in \cite{Segura}.
\vskip 1em
\begin{theorem}
Suppose that $t>0$ and $\nu >0$ are real numbers.  Then
\begin{equation}
\frac{K_{\nu+1}(t)}{K_{\nu}(t)} < \frac{\nu + \sqrt{\nu^2+t^2}}{t} 
\leq \frac{2\nu}{t}+1.
\end{equation}
\label{preliminaries:bessel:theorem1}
\end{theorem}

\label{section:preliminaries:bessel}
\end{subsection}

%
%

\begin{subsection}{The binomial theorem}

A proof of the following can be found in \cite{Rudin}, as well as 
many other sources.

\vskip 1em
\begin{theorem}
Suppose that $r$ is a real number,  and that
 $y$ is a  real number such that $|y| < 1$.  Then
\begin{equation}
(1+y)^r = \sum_{k=0}^\infty  \frac{\Gamma(r+1)}{\Gamma(k+1)\Gamma(r-k+1)} y^k.
\end{equation}
%
%
\label{preliminaries:binomial_theorem}
\end{theorem}
\label{section:preliminaries:binomial}
\end{subsection}

%
%
\begin{subsection}{The Lambert {$W$} function}
The Lambert $W$ function or product logarithm is the multiple-valued
inverse of the function
\begin{equation}
f(z) = z \exp(z).
\end{equation}
We follow \cite{knuth} in using $W_{0}$ to denote 
the branch of $W$ which is real-valued and greater than or equal to 
$-1$ on the interval $[-1/e,-\infty)$ and
$W_{-1}$ to refer to
the branch which is real-valued and less than or equal to $-1$ on $[-1/e,0)$.

We will make use of the following elementary facts concerning $W_0$ and $W_{-1}$
(all of which can be found in \cite{knuth} and its references).
\vskip 1em
\begin{theorem}
Suppose that $y \geq -1/e$ is a real number.  Then
\begin{equation}
x\exp(x) \leq y 
\ \ \ \mbox{if and only if}\ \ \
x \leq W_0(y).
\end{equation}
\label{preliminaries:LambertW:theorem1}
\end{theorem}

\begin{theorem}
Suppose that $0 < y \leq 1/e$ is a real number.  Then
\begin{equation}
x\exp(-x) \leq y
\ \ \ \mbox{if and only if}\ \ \
x \geq -W_{-1}(-y).
\end{equation}
\label{preliminaries:LambertW:theorem2}
\end{theorem}

\begin{theorem}
Suppose that $0 \leq x \leq 1$ is a real number.  Then
\begin{equation}
\frac{x}{2} \leq W_0(x) \leq x.
\end{equation}
\label{preliminaries:LambertW:theorem3}
\end{theorem}

\begin{theorem}
Suppose that $x > \exp(1)$ is a real number.  Then
\begin{equation}
\log(x) \leq -W_{-1}\left(-\frac{1}{x}\right) \leq 2 \log(x).
\end{equation}
\label{preliminaries:LambertW:theorem4}
\end{theorem}

%



\label{section:preliminaries:LambertW}
\end{subsection}

\begin{subsection}{Fr\'echet derivatives and the contraction mapping principle}

Given Banach spaces $X$, $Y$ and a mapping $f:X \to Y$ between them, we say
that $f$ is Fr\'echet differentiable at $x \in X$ if there exists a 
bounded linear operator  
$X \to Y$, denoted by $f_x'$, such that
\begin{equation}
\lim_{h \to 0} \frac{\left\|f(x+h)-f(x)-f_x'\left[h\right]\right\|}{\|h\|} = 0.
\end{equation}

\vskip 1em
\begin{theorem}
Suppose that $X$ and $Y$ are a Banach spaces and that $f:X \to Y$ is Fr\'echet differentiable
at every point of $X$.  Suppose also that $D$ is a convex subset of $X$, and
that there exists a real number $M > 0$ such that
\begin{equation}
\|f'_x\| \leq M
\label{preliminaries:frechet:1}
\end{equation}
for all $x \in D$.   Then
\begin{equation}
\|f(x) - f(y)\| \leq M \|x-y\|
\end{equation}
for all $x$ and $y$ in $D$.
\label{preliminaries:frechet:mvt}
\end{theorem}

Suppose that $f:X \to X$ is a mapping of the Banach space $X$ into itself.
We say that $f$ is contractive on a subset $D$ of $X$  if 
there exists a real number $0 < \alpha < 1$ such that 
\begin{equation}
\|f(x)-f(y)\| \leq \alpha \|x-y\|
\end{equation}
for all $x,y \in D$.
Moreover, we say that  $\{x_n\}_{n=0}^\infty$ is a sequence
of fixed point iterates for $f$ if $x_{n+1} = f(x_n)$ for all $n \geq 0$.

Theorem~\ref{preliminaries:frechet:mvt} is often used to show
that a mapping is contractive so that the 
following result can be applied.

\vskip 1em
\begin{theorem}{(The Contraction Mapping Principle)}
Suppose that $D$ is a closed subset of a Banach space $X$.
Suppose also that $f:X \to X$ is contractive on $D$ and
$f(D) \subset D$.
Then the equation
\begin{equation}
x = f(x)
\label{frechet:1}
\end{equation}
has a unique solution $\sigma^* \in D$.    Moreover, any sequence of
 fixed point iterates for the function $f$
which contains an element in $D$ 
 converges to $\sigma^*$.
\label{preliminaries:frechet:cmp}
\end{theorem}

A discussion of Fr\'echet derivatives and 
proofs of Theorems~\ref{preliminaries:frechet:mvt} and
\ref{preliminaries:frechet:cmp} can be found, for instance, in \cite{ZeidlerI}.

\label{section:preliminaries:frechet}
\end{subsection}

%
%
\begin{subsection}{Gronwall's inequality}
The following well-known inequality can be found in, for instance,  \cite{Bellman}.
\vskip 1em
\begin{theorem}
Suppose that $f$ and $g$ are continuous functions on the interval $[a,b]$ such that
\begin{equation}
f(t) \geq 0 \ \ \mbox{and}\ \ g(t) \geq 0 \ \ \ \mbox{for all}\ \ a \leq t \leq b.
\end{equation}
Suppose further that there exists a real number $C>0$ such that
\begin{equation}
f(t) \leq C + \int_s^t f(s)g(s)\ ds \ \ \ \mbox{for all}\ \ a \leq t \leq b.
\end{equation}
Then
\begin{equation}
f(t) \leq C \exp\left(\int_a^t g(s)\ ds\right)\ \ \ \mbox{for all}\ \ a \leq t \leq b.
\end{equation}
\label{preliminaries:theorem:gronwall}
\end{theorem}
\label{preliminaries:section:gronwall}
\end{subsection}

%
%
\begin{subsection}{Schwarzian derivatives}
The Schwarzian derivative of a smooth function $f: \mathbb{R} \to \mathbb{R}$ 
is
\begin{equation}
\{f,t\} = \frac{f'''(t)}{f'(t)} - \frac{3}{2} \left(\frac{f''(t)}{f'(t)}\right)^2.
\end{equation}
If the function $x(t)$ is a diffeomorphism of the real line (that is, a smooth,
invertible mapping $\mathbb{R} \to \mathbb{R}$),
then the Schwarzian derivative of $x(t)$
can be related to the Schwarzian derivative of its inverse $t(x)$;  in particular,
\begin{equation}
\{x,t\} = - \left(\frac{dx}{dt}\right)^2\{t,x\}.
\label{preliminaries:Schwarzian_derivative:change_of_vars}
\end{equation}
The identity~(\ref{preliminaries:Schwarzian_derivative:change_of_vars})
can be found, for instance, in Section 1.13 of \cite{NISTHandbook}.

\label{section:preliminaries:schwarzian_derivative}
\end{subsection}

\label{section:preliminaries}
\end{section}

 \begin{section}{Integral equation formulation}

In this section, we reformulate Kummer's equation
\begin{equation}
\left(\alpha'(t)\right)^2 = \lambda^2q(t) - \frac{1}{2}\frac{\alpha'''(t)}{\alpha'(t)}
+ \frac{3}{4} \left(\frac{\alpha''(t)}{\alpha'(t)}\right)^2
\label{inteq:kummers_equation}
\end{equation}
as a nonlinear integral equation.   As in the introduction, we 
assume that the function $q$ has been extended to the real line
and we seek a function $\alpha$ which satisfies (\ref{inteq:kummers_equation})
on the real line.

By letting
\begin{equation}
\left(\alpha'(t)\right)^2 = \lambda^2 \exp(r(t))
\end{equation}
in (\ref{inteq:kummers_equation}), we obtain  the equation
\begin{equation}
r''(t) - \frac{1}{4}\left(r'(t)\right)^2 + 4 \lambda^2\left( \exp(r(t)) - q(t)\right) = 0
\ \ \ \mbox{for all}\ \ \ t\in\mathbb{R}.
\label{inteq:kummer_logarithm_form}
\end{equation}
We next take $r$ to be of the form
\begin{equation}
r(t) = \log(q(t)) + \delta(t),
\end{equation}
which results in 
\begin{equation}
\delta''(t) - \frac{1}{2}\frac{q'(t)}{q(t)}\delta'(t) 
-\frac{1}{4}\left(\delta'(t)\right)^2
+ 4\lambda^2 
q(t) \left(\exp(\delta(t))-1\right) 
=
q(t) p(t),
\ \ \ \mbox{for all}\ \ \ t\in\mathbb{R},
\label{inteq:delta_equation}
\end{equation}
where $p$ is defined by the formula
\begin{equation}
p(t) = \frac{1}{q(t)} \left(
\frac{5}{4} \left(\frac{q'(t)}{q(t)}\right)^2
-\frac{q''(t)}{q(t)} \right).
\label{inteq:definition_of_p}
\end{equation}
Note that the function $p$ appears in the standard
error analysis of WKB approximations (see, for instance, \cite{Olver}).
Expanding the exponential in a power series and rearranging terms yields
the equation
\begin{equation}
\delta''(t) - \frac{1}{2}\frac{q'(t)}{q(t)}\delta'(t) + 4 \lambda^2 q(t) \delta(t)
- \frac{1}{4} \left(\delta'(t)\right)^2 + 4\lambda^2 q(t)
 \left(\frac{\left(\delta(t)\right)^2}{2} + 
\frac{\left(\delta(t)\right)^3}{3!} + \cdots \right) 
= q(t) p(t).
\label{inteq:differential}
\end{equation}
 Applying the change of variables
\begin{equation}
x(t) = \int_0^t \sqrt{q(u)}\ du
\label{inteq:changeofvars}
\end{equation}
transforms (\ref{inteq:differential}) into
\begin{equation}
\delta''(x) + 4\lambda^2 \delta(x) 
-\frac{1}{4} \left(\delta'(x)\right)^2
+ 4\lambda^2 \left(\frac{\left(\delta(x)\right)^2}{2} + 
\frac{\left(\delta(x)\right)^3}{3!} + \cdots \right) = p(x)
\ \ \ \mbox{for all}\ \ \ x \in \mathbb{R}.
\label{inteq:differential_changeofvars}
\end{equation}

At first glance, the relationship between
the function $p(x)$ appearing in (\ref{inteq:differential_changeofvars})
and the coefficient $q(t)$ in the ordinary differential equation
(\ref{introduction:original_equation}) is complex.
However,  the function $p(t)$ defined via (\ref{inteq:definition_of_p}) is related
to the Schwarzian derivative  (see Section~\ref{section:preliminaries:schwarzian_derivative}) 
of the function $x(t)$ defined in (\ref{inteq:changeofvars}) via
the formula
\begin{equation}
p(t) = - \frac{2}{q(t)} \left\{x,t\right\}
=
-2 \left(\frac{dt}{dx}\right)^2 \left\{x,t\right\}.
\label{inteq:pschwarz}
\end{equation}
It follows from (\ref{inteq:pschwarz}) and 
Formula~(\ref{preliminaries:Schwarzian_derivative:change_of_vars}) 
in Section~\ref{section:preliminaries:schwarzian_derivative}
that 
\begin{equation}
p(x) = 2 \left\{t,x\right\}.
\label{inteq:pschwarz2}
\end{equation}
That is to say: $p$, when viewed as a function of $x$, is simply
twice the Schwarzian derivative of $t$ with respect to $x$.

It is also notable that the part of (\ref{inteq:differential_changeofvars}) 
which is linear in $\delta$ is a constant coefficient  Helmholtz equation.  
This suggests that  we form an  integral equation for (\ref{inteq:differential_changeofvars})
using a Green's function for the Helmholtz equation.
To that end, we define the integral operator
$T$  
for functions $f \in \Lp{1}$ via the formula
\begin{equation}
\begin{aligned}
\OpT{f}(x) &= 
\frac{1}{4\lambda}
\int_{-\infty}^\infty
\sin\left(2 \lambda \left|x-y\right|\right) f(y)\ dy
\end{aligned}
\label{inteq:definition_of_T}
\end{equation}
The following theorem summarizes the relevant properties of the operator
$T$.  
%
\vskip 1em
\begin{theorem}
Suppose that $\lambda >0$ is a real number, and that the
operator $T$ is defined as in (\ref{inteq:definition_of_T}).
Suppose also that   $f \in \Lp{1} \cap  C\left(\mathbb{R}\right)$.  Then:
\begin{enumerate}
\item
$\OpT{f}(x)$ is an element of $\C{2}$;
\item
 $\OpT{f}(x)$ is a solution of the  ordinary differential equation
\begin{equation*}
y''(x) + 4 \lambda^2 y(x) = f(x)\ \ \  \mbox{for all}\ x \in \mathbb{R};\ \mbox{and}
\end{equation*}
\item
the Fourier transform of $\OpT{f}(x)$ is the principal value of 
\begin{equation*}
\frac{\widehat{f}(\xi)}{4\lambda^2-\xi^2}
=
\frac{1}{4\lambda}
\left(
\frac{\widehat{f}(\xi)}{2\lambda-\xi}
+
\frac{\widehat{f}(\xi)}{2\lambda+\xi}
\right).
\end{equation*}
\end{enumerate}
\label{inteq:TTheorem}
\end{theorem}

%
\begin{proof}
We observe that
\begin{equation}
\begin{aligned}
\OpT{f}(x) 
= 
&\frac{1}{4\lambda}
\int_{-\infty}^x
\sin\left(2 \lambda \left(x-y\right)\right) f(y)\ dy
+
\frac{1}{4\lambda}
\int_{x}^{\infty}
\sin\left(2 \lambda (y-x)\right) f(y)\ dy
\\
=
&
\frac{1}{4\lambda}
\sin(2\lambda x)
\int_{-\infty}^x
 \cos(2\lambda y) f(y)\ dy
-
\frac{1}{4\lambda}
\cos(2\lambda x)
 \int_{-\infty}^x
\sin(2\lambda y) f(y)\ dy 
\\
+&
\frac{1}{4\lambda}
\cos(2\lambda x)
\int_{x}^{\infty}
\sin(2\lambda y ) 
f(y)\ dy
-
\frac{1}{4\lambda}
\sin(2\lambda x)
\int_{x}^{\infty}
\cos(2\lambda y)
f(y)\ dy
\end{aligned}
\label{inteq:theorem:1}
\end{equation}
for all $x \in \mathbb{R}$.   
We differentiate both sides of (\ref{inteq:theorem:1}) with respect to $x$,
apply the Lebesgue dominated convergence theorem to each integral
(this is permissible since the sine and cosine functions are bounded
and $f \in \Lp{1}$) and combine terms in order to conclude that
 $\OpT{f}$ is differentiable  everywhere and
\begin{equation}
\begin{aligned}
\frac{d}{dx} 
\OpT{f}(x) 
=
&=
\frac{1}{2}
\int_{-\infty}^{\infty}
\cos\left(2 \lambda \left|x-y\right|\right) \sign(x-y) f(y)\ dy
\end{aligned}
\label{inteq:theorem:2}
\end{equation}
for all $x \in \mathbb{R}$.   In the same fashion, we conclude that
\begin{equation}
\begin{aligned}
\left(\frac{d}{dx}\right)^2
\OpT{f}(x) 
&=
f(x)
-\lambda
\int_{-\infty}^{\infty}
\sin\left(2 \lambda \left|x-y\right|\right) f(y)\ dy
\end{aligned}
\label{inteq:theorem:3}
\end{equation}
for all $x \in \mathbb{R}$.  
Since $f$ is continuous by assumption and the second term
appearing on the right-hand side in (\ref{inteq:theorem:3})
is a continuous function of $x$ by the Lebesgue
dominated convergence theorem, we see from  (\ref{inteq:theorem:3})
that $\OpT{f}$ is twice continuously differentiable.
By combining (\ref{inteq:theorem:3})  and  (\ref{inteq:definition_of_T}),
we conclude that  $\OpT{f}$ is a solution of the ordinary differential
equation
\begin{equation}
y''(x) + 4 \lambda^2 y(x) = f(x)\ \ \  \mbox{for all}\ x \in \mathbb{R}.
\end{equation}
%

We now define the function $g$ through the formula
\begin{equation}
\widehat{g}(\xi) = 
\frac{1}{4\lambda}\left(\frac{1}{2\lambda - \xi }  + \frac{1}{2\lambda + \xi}\right).
\label{inteq:TTheorem:4}
\end{equation}
It is well known that the Fourier transform of the principal value
of $1/x$ is the function 
\begin{equation}
 -i \pi \sign(x);
\end{equation}
see, for instance, \cite{Stein-Weiss} or \cite{GrafakosC}.  
It follows readily  that the inverse Fourier transform
of the principal value of 
\begin{equation}
 \frac{1}{2\lambda \pm \xi}
\end{equation}
is
\begin{equation}
\pm \frac{1}{2i} \exp\left(\mp 2\lambda i x\right)\sign(x).
\end{equation}
From this and (\ref{inteq:TTheorem:4}), we conclude that
%
\begin{equation}
\begin{aligned}
g(x)
&= 
\frac{1}{4\lambda}\left(
 \frac{1}{2i} \exp\left(-2\lambda i x\right)\sign(x)
-\frac{1}{2i} \exp\left(2\lambda i x\right)\sign(x)\right)
\\
&=
\frac{1}{4\lambda} \sin\left(2\lambda|x|\right).
\end{aligned}
\label{inteq:Ttheorem:5}
\end{equation}
In particular, $\OpT{f}$ is the convolution of $f$ with $g$.
As a consequence, 
\begin{equation}
\widehat{\OpT{f}}(\xi) = \widehat{g}(\xi) \widehat{f}(\xi) = 
\frac{1}{4\lambda}
\left(
\frac{\widehat{f}(\xi)}{2\lambda-\xi}
+
\frac{\widehat{f}(\xi)}{2\lambda+\xi}
\right),
\end{equation}
which is the third and final conclusion of the theorem.
\end{proof}

In light of Theorem~\ref{inteq:TTheorem}, it is clear that introducing the
representation
\begin{equation}
\delta(x) = \OpT{\sigma}(x)
\end{equation}
into (\ref{inteq:differential_changeofvars}) yields the nonlinear integral
equation
\begin{equation}
\sigma(x)  =
\OpS{\OpT{\sigma}}(x) + p(x)
\ \ \ \mbox{for all}\ \ x\in\mathbb{R},
\label{inteq:integral_equation}
\end{equation}
where $S$ is the operator defined for functions $f \in C^1\left(\mathbb{R}\right)$ by the
formula
\begin{equation}
\begin{aligned}
\OpS{f}(x) 
&= 
\frac{\left(f'(x)\right)^2}{4}
- 4 \lambda^2 \left(\frac{\left(f(x)\right)^2}{2!} + 
\frac{\left(f(x)\right)^3}{3!} + \frac{\left(f(x)\right)^4}{4!}+ \cdots \right).
\end{aligned}
\label{inteq:definition_of_S}
\end{equation}

The following theorem is immediately apparent from the  procedure 
used to transform Kummer's equation (\ref{inteq:kummers_equation}) into
the nonlinear integral equation (\ref{inteq:integral_equation}).
\vskip 1em
\begin{theorem}
Suppose that $\lambda > 0$ is a real number, that $q:\mathbb{R} \to \mathbb{R}$
is an infinitely differentiable, strictly positive
function,
that $x(t)$ is  defined by (\ref{inteq:changeofvars}), and that $p(x)$ is defined
via (\ref{inteq:pschwarz2}).  
Suppose also that  $\sigma \in \Lp{1} \cap C\left(\mathbb{R}\right)$  is a 
solution of the integral 
equation (\ref{inteq:integral_equation}), that
$\delta$ is defined via the formula
\begin{equation}
\delta(x) = \OpT{\sigma}(x) 
= \frac{1}{4\lambda} \int_{-\infty}^{\infty}
\sin\left(2\lambda\left|x-y\right|\right) \sigma(y)\ dy,
\label{inteq:theorem:delta}
\end{equation}
and that the function $\alpha$ is defined by the formula
\begin{equation}
\alpha(t) = \lambda \int_0^t  \sqrt{q(u) }
\exp\left(\frac{\delta(x(u))}{2}\right)\ du.
\label{inteq:theorem:alpha}
\end{equation}

Then:  

\begin{enumerate}
\item
$\delta(x)$ is a twice continuously differentiable solution
of (\ref{inteq:differential_changeofvars});

\item
$\delta(x(t))$  is a twice continuously differentiable solution of of (\ref{inteq:differential});
\item
$\alpha(t)$ is three times continuously differentiable solution of 
(\ref{inteq:kummers_equation}); and
\item
 $\alpha(t)$  is a phase function for the ordinary differential equation
\begin{equation}
y''(t) + \lambda^2 q(t) y(t) = 0\ \ \ \mbox{for all}\ \ 0 \leq t \leq 1.
\end{equation}
\end{enumerate}
\label{inteq:theorem}
\end{theorem}

\label{section:inteq}
\end{section}

\begin{section}{Overview and statement of the principal result}

The composition operator $S\circ T$ appearing 
in (\ref{inteq:integral_equation}) does not map any Lebesgue space
$\Lp{p}$  or H\"older space $C^{k,\alpha}(\mathbb{R})$ to itself, which
complicates  the analysis of (\ref{inteq:integral_equation}).
Moreover, the integral  defining  $\OpT{\sigma}$ only exists if either $\sigma \in \Lp{1}$ or 
$\widehat{\sigma}(\pm 2\lambda) = 0$.  
Even if both of these  conditions are satisfied, it is not necessarily
the case that  $\OpS{\OpT{\sigma}}  + p$ will satisfy either condition.
This casts doubts on whether (\ref{inteq:integral_equation})
is solvable for arbitrary $p$.

We avoid a detailed discussion of which spaces and in what sense
 (\ref{inteq:integral_equation}) admits solutions
and instead show that, under mild conditions on $p$ and $\lambda$, there exists a function
$p_b$ which   approximates $p$ and such that the equation
\begin{equation}
\sigma(x)  =
\OpS{\OpT{\sigma}}(x) + p_b(x)
\ \ \ \mbox{for all}\ \ x\in\mathbb{R}
\label{overview:inteq}
\end{equation}
admits a solution.    
Moreover, we prove that if $p$ is nonoscillatory
then the solution of (\ref{overview:inteq}) 
is also nonoscillatory
 and $\|p-p_b\|_\infty$ decays exponentially in $\lambda$.  
The next theorem, which is the principal result of this article, makes
these statements precise.  Its proof is given in Sections~\ref{section:bandlimit},
\ref{section:bandlimit:solution}, \ref{section:spectrum} and \ref{section:solution}.

%
%

\vskip 1em
\begin{theorem}
Suppose that  $q \in C^\infty\left(\mathbb{R}\right)$ is a strictly positive,
and that $x(t)$ is defined by the formula
\begin{equation}
x(t) = \int_0^t \sqrt{q(u)}\ du.
\label{overview:x}
\end{equation}
Suppose also that $p(x)$ is defined via the formula 
\begin{equation}
p(x) = 2\{t,x\};
\end{equation}
that is, $p(x)$ is twice the Schwarzian derivative of the
variable $t$ with respect to the variable $x$ defined via
 (\ref{overview:x}).
Suppose furthermore that there exist real numbers
$\lambda>0$, $\Gamma >0$ and $a >0$ such that
\begin{equation}
\lambda \geq 6 \max\left\{\frac{1}{a},\Gamma\right\}
\end{equation}
and
\begin{equation}
\left|\widehat{p}(\xi)\right| \leq 
\Gamma \exp\left(-a\left|\xi\right|\right)
\ \ \ \mbox{for all}\ \ \xi\in\mathbb{R}.
\label{over:condition_on_phat}
\end{equation}
Then there exist functions $p_b\in C^{\infty}\left(\mathbb{R}\right)$
and $\sigma_b \in \Lp{2} \cap C^{\infty}\left(\mathbb{R}\right)$
such that
$\sigma_b$ is a solution of (\ref{overview:inteq}),
\begin{equation}
\left|\widehat{\sigma_b}(\xi)\right| \leq 
2\Gamma
\exp\left(-\frac{5}{6}a|\xi|\right) 
\ \ \ \mbox{for all}\ \ \xi\in\mathbb{R},
\end{equation}
and
\begin{equation}
\|p-p_b\|_\infty < \frac{24 \Gamma}{5a} \exp\left(-\frac{5}{6} a \lambda\right).
\end{equation}
\label{main_theorem}
\end{theorem}

According to Theorem~\ref{inteq:theorem}, if $\sigma$ is a solution 
of the integral equation (\ref{inteq:integral_equation}),
then the function $\alpha$ defined by  (\ref{inteq:theorem:alpha}) is a 
phase function for the differential equation (\ref{introduction:original_equation}).
We define $\alpha_b$ in analogy with (\ref{inteq:theorem:alpha}) using the solution $\sigma_b$ of
the modified equation (\ref{overview:inteq}).
That is, we let $\delta_b$ be the function defined by the formula
\begin{equation}
\delta_b(x) = \OpT{\sigma_b}(x) = \frac{1}{4\lambda} \int_{-\infty}^{\infty} 
\sin(2\lambda|x-y|) \sigma_b(y)\ dy,
\end{equation}
and then define $\alpha_b$ via
\begin{equation}
\alpha_b(t) = \lambda \int_0^t  \sqrt{q(u) }
\exp\left(\frac{\delta_b(x(u))}{2}\right)\ du.
\end{equation}
Obviously, $\alpha_b$ is not a phase function for the equation
(\ref{introduction:original_equation}).  However, if
we define  $q_b:\mathbb{R}\to\mathbb{R}$ by the formula
\begin{equation}
 \frac{1}{q_b(t)} \left(
\frac{5}{4} \left(\frac{q_b'(t)}{q_b(t)}\right)^2
-\frac{q_b''(t)}{q_b(t)} \right) = p_b(t),
\label{overview:q0p0}
\end{equation}
then Theorem~\ref{inteq:theorem}
implies that 
$\alpha_b$ is a phase function for the differential equation
\begin{equation}
y''(t) + \lambda^2 q_b(t) y(t) = 0\ \ \ \mbox{for all}\ \ 0 \leq t \leq 1.
\label{overview:approximate_equation}
\end{equation}
Since  $\|q-q_b\|_\infty$ is bounded in terms of 
$\|p-p_b\|_\infty$
and the solutions of (\ref{overview:approximate_equation})
closely approximate those of  (\ref{introduction:original_equation}) 
when $\|q-q_b\|_\infty$ is small, the function $\alpha_b$
can be used to approximation solutions of 
(\ref{introduction:original_equation}) 
when $\|p-p_b\|_\infty$ is small.  

Theorem~\ref{overview:backwards}, which appears below, gives a relevant error
estimate.  Given $\epsilon >0$, it specifies a bound on $\|p-p_b\|_\infty$ 
which ensures that solutions of (\ref{overview:approximate_equation})
approximate those of (\ref{introduction:original_equation}) to relative
precision $\epsilon$.  Its proof appears in Section~\ref{section:backwards}.

\vskip 1em
\begin{definition}
We say that $\alpha$ is an $\epsilon$-approximate
phase function for the ordinary differential equation (\ref{introduction:original_equation})
if there exists a basis of solutions $\left\{\widetilde{u},\widetilde{v}\right\}$ of
(\ref{introduction:original_equation}) such that
\begin{equation}
\left|u(t) - \widetilde{u}(t)\right| \leq
\epsilon \sup_{0 \leq t \leq 1} \left|\widetilde{u}(t)\right|
\ \ \
\mbox{for all} \ \ 0 \leq t \leq 1
\end{equation}
and
\begin{equation}
\left|v(t) - \widetilde{v}(t)\right| \leq
\epsilon \sup_{0 \leq t \leq 1} \left|\widetilde{v}(t)\right|
\ \ \
\mbox{for all} \ \ 0 \leq t \leq 1,
\end{equation}
where $u$, $v$ are defined by  
%
\begin{equation}
u(t) = \frac{\sin(\alpha(t)) }{\left|\alpha'(t)\right|^{1/2}}
\end{equation}
and
\begin{equation}
v(t) = \frac{\cos(\alpha(t)) }{\left|\alpha'(t)\right|^{1/2}}.
\end{equation}
\end{definition}

\vskip 1em
\begin{theorem}
Suppose that
$q \in C^\infty\left(\mathbb{R}\right)$ is strictly positive,
 that the function $p$ is defined by the formula (\ref{inteq:definition_of_p}),
and that  there exist real numbers $0 < \eta_1 < \eta_2 $ such that
\begin{equation}
\eta_1 \leq q(t) \leq \eta_2
\ \ \ \mbox{for all} \ \ 0 \leq t \leq 1,
\label{overview:backwards1}
\end{equation}
\begin{equation}
\left|\frac{q'(t)}{q(t)}\right| \leq \eta_2
\ \ \ \mbox{for all}\ \  0 \leq t \leq 1,
\label{overview:backwards2}
\end{equation}
and 
\begin{equation}
\left|p(t)\right| \leq \eta_2
\ \ \ \mbox{for all}\ \  0 \leq t \leq 1.
\label{overview:backwards3}
\end{equation}
Suppose also that  $\lambda>0$, $\epsilon >0$ are real numbers
such that
\begin{equation}
0 < \epsilon < \lambda
\exp\left(\frac{\eta_2^{3/4}}{4}\right),
\label{overview:backwards4}
\end{equation}
and that $k$ is the real number defined by 
\begin{equation}
k=20 \left(\frac{\eta_2}{\eta_1}\right)^2 + 8 \eta_2^2 + 10 \frac{\eta_2}{\eta_1} + 1.
\end{equation}
Suppose furthermore that  $p_b:\mathbb{R} \to \mathbb{R}$
is an infinitely differentiable function such that
\begin{equation}
\|p-p_b\|_\infty
\leq \frac{1}{2}
\frac{\eta_1 }{\lambda}
\exp\left(-k\right)
\exp\left(-\frac{\eta_2^{3/4}}{4}\right)
\epsilon.
\end{equation}
Then there exists a function $q_b:[0,1]\to\mathbb{R}$  such that
(\ref{overview:q0p0}) holds and any phase function for (\ref{overview:approximate_equation})
is an $\epsilon$-approximate phase function for (\ref{introduction:original_equation}).
\label{overview:backwards}
\end{theorem}

By combining Theorems~\ref{inteq:theorem}, \ref{main_theorem} and
\ref{overview:backwards} we obtain the following.
%
\vskip 1em
\begin{corollary}
Suppose that $q$, $x(t)$, $p$, $\eta_1$, $\eta_2$, $\lambda$, $\epsilon$, $k$,
$\Gamma$, $a$ and $\eta$ are as in
the hypotheses of Theorem~\ref{main_theorem} and \ref{overview:backwards}.
Suppose that, in addition,  %
\begin{equation}
\lambda \geq
\frac{6}{5a}\left(
k + \frac{\eta_2^{3/4}}{4} + 
 \log\left(
 \frac{12\Gamma}{\eta_1 a^2}
\frac{1}{\epsilon}
\right)
\right).
\label{overview:corollary:1}
\end{equation}
Then there exists a function $\sigma_b \in \Lp{2}\cap C^\infty\left(\mathbb{R}\right)$ 
such that
\begin{equation}
\left|\widehat{\sigma_b}(\xi)\right| \leq 
2\Gamma
\exp\left(-\frac{5}{6}a|\xi|\right) 
\ \ \ \mbox{for all}\ \ \xi \in \mathbb{R}
\end{equation}
and the function $\alpha_b(t)$ defined by 
\begin{equation}
\alpha_b(t) = \lambda \int_0^t  \sqrt{q(u) }
\exp\left(\frac{\delta_b(x(u))}{2}\right)\ du,
\end{equation}
where $\delta_b$ is the function defined through the formula
\begin{equation}
\delta_b(x)  = \frac{1}{4\lambda} \int \sin(2\lambda|x-y|) \sigma_b(y)\ dy,
\end{equation}
is an $\epsilon$-approximate phase function for the ordinary differential
equation (\ref{introduction:original_equation}).
\end{corollary}

\vskip 1em
\begin{remark}

The hypothesis (\ref{overview:corollary:1}) can be replaced with
the weaker condition
\begin{equation}
\lambda >  
\frac{6}{5a} 
W_{-1}\left(-
\frac{25}{288} \frac{\eta_1 a^2}{\Gamma}
\exp\left(-k\right)\exp\left(-\frac{\eta_2^{3/4}}{4}\right)
\epsilon
\right),
\end{equation}
where $W_{-1}$ denotes the branch of the Lambert $W$ function which
is real-valued and less than $-1$ on the interval $[-1/e,0)$
(see Section~\ref{section:preliminaries:LambertW}).
\end{remark}

The proof of Theorem~\ref{main_theorem} is divided amongst Sections~\ref{section:bandlimit},
\ref{section:bandlimit:solution}, \ref{section:spectrum} and \ref{section:solution}.   
The principal difficulty lies in  constructing
a function $p_b$ which approximates $p$ and for which (\ref{overview:inteq}) admits a solution.
We accomplish this by introducing a modified integral equation
\begin{equation}
\sigma(x) = \OpS{\OpTb{\sigma}}(x) + p(x),
\label{overview:modinteq}
\end{equation}
where $T_b$ is a ``band-limited'' version of $T$.  That is, $\OpTb{f}$ is defined
via the formula
\begin{equation}
\widehat{\OpTb{f}}(\xi) = \widehat{\OpT{f}}(\xi) b(\xi),
\end{equation}
where $b(\xi)$ is a $C_c^{\infty}\left(\mathbb{R}\right)$ bump function.
This modified integral equation is introduced in Section~\ref{section:bandlimit}.

In Section~\ref{section:bandlimit:solution}, we show that under 
mild conditions on $p$ and $\lambda$,  (\ref{overview:modinteq})
admits a solution $\sigma$.
The argument proceeds by applying the Fourier transform to (\ref{overview:modinteq})
and using the contraction mapping principle to show that the
resulting equation admits a  solution.

In Section~\ref{section:spectrum},  we estimate the Fourier transform of the solution 
$\sigma$ of (\ref{overview:modinteq})
under the assumption that $\widehat{p}$ is exponentially decaying.  
We show that $\widehat{\sigma}$ is also exponentially decaying, albeit at a slightly
slower rate.  

In Section~\ref{section:solution}, we define a band-limited version $\sigma_b$ of 
the solution $\sigma$ of (\ref{overview:modinteq}) via the formula
\begin{equation}
\widehat{\sigma_b}(\xi) = \widehat{\sigma}(\xi) b(\xi).
\end{equation}
By combining the observation that  $\OpTb{\sigma} = \OpT{\sigma_b}$ with
(\ref{overview:modinteq}), we obtain 
\begin{equation}
\sigma(x) = \OpS{\OpT{\sigma_b}}(x) + p(x).
\label{overview:12}
\end{equation}
We then define $p_b$ by the formula
\begin{equation}
p_b(x) = p(x) + \sigma_b(x) - \sigma(x)
\end{equation}
and rearrange (\ref{overview:12}) as
\begin{equation}
\sigma_b(x) = \OpS{\OpT{\sigma_b}}(x) + p_b(x).
\end{equation}
%
The decay estimate on $\widehat{\sigma}$ is then used to bound 
$\|p-p_b\|_\infty=\|\sigma-\sigma_b\|_\infty \leq \|\widehat{\sigma}-\widehat{\sigma_b}\|_1$.






\label{section:overview}
\end{section}

\begin{section}{Band-limited integral equation}

In this section, we introduce a ``band-limited'' 
version of the operator $T$ and use it  to form an alternative
to the integral equation (\ref{inteq:integral_equation}).

Let  $b(\xi)$  be any infinitely  differentiable function such that
\begin{enumerate}
\item
$
\begin{aligned}
b(\xi) = 1 \ \ &\mbox{for all}\ \ |\xi| \leq \lambda,
\end{aligned}
$

\item
$
\begin{aligned}
0 \leq b(\xi) \leq 1  \ \ &\mbox{for all}\ \ \lambda \leq |\xi| \leq \sqrt{2}\lambda,
\ \ \mbox{and} 
\end{aligned}
$

\item
$
\begin{aligned}
b(\xi) = 0 \ \ &\mbox{for all}\ \ |\xi| \geq \sqrt{2}\lambda.
\end{aligned}
$
\end{enumerate}
We define $\OpTb{f}$ for functions $f \in \Lp{1}$ via the formula
\begin{equation}
\widehat{\OpTb{f}}(\xi) 
= 
\widehat{f}(\xi)\frac{b(\xi)}{4\lambda^2-\xi^2}
\label{bandlimit:definition_of_Tb}
\end{equation}
We will refer to $T_b$ as the band-limited version of the operator $T$ and 
and we call the nonlinear integral equation
\begin{equation}
\sigma(x) = 
\OpS{\OpTb{\sigma}}(x) + p(x)
\ \ \ \mbox{for all}\ \ x\in\mathbb{R}
\label{bandlimit:integral_equation}
\end{equation}
obtained by replacing $T$ with $T_b$ in (\ref{inteq:integral_equation})  
the ``band-limited'' version of  (\ref{inteq:integral_equation}).  

Since $T_b$ is a  Fourier multiplier, it is convenient to analyze 
(\ref{bandlimit:integral_equation}) in the Fourier domain rather than the space domain.
We now introduce notation which will allow us to write down the 
equation obtained by applying the Fourier transform to both sides
of (\ref{bandlimit:integral_equation}).

We let  $W_b$ and $\widetilde{W}_b$  be the linear operators defined 
for $f \in \Lp{1}$  via the formulas
\begin{equation}
\OpWb{f}(\xi) = 
f(\xi)
\frac{b(\xi) }{4\lambda^2-\xi^2}
\label{bandlimit:definition_of_W}
\end{equation}
and
\begin{equation}
\OpWTb{f}(\xi) = 
f(\xi)
\frac{b(\xi) i \xi }{4\lambda^2-\xi^2},
\label{bandlimit:definition_of_W'}
\end{equation}
where $b(\xi)$ is the function used to define the operator
$T_b$.  

For functions $f \in \Lp{1}$, it is standard to denote 
the Fourier transform of the function $\exp(f(x))$ by 
$\exp^*\left[f\right]$; that is,
\begin{equation}
\exp^*\left[f\right](\xi) = \delta(\xi) + 
f(\xi) + \frac{f*f(\xi)}{2!} + \frac{f*f*f(\xi)}{3!} + \cdots.
\label{bandlimit:exp*}
\end{equation}
In (\ref{bandlimit:exp*}), $\delta$ defers to the delta distribution  and
$f*f*\cdots*f$ denotes repeated convolution of the function $f$ with itself.  
The Fourier transform of $\exp(f(x))$ never appears in this paper;
however, we  will encounter the Fourier transforms of the functions
\begin{equation}
\exp(f(x)) - 1
\end{equation}
and
\begin{equation}
\exp(f(x)) - f(x) - 1.
\end{equation}
So, in analogy with the definition (\ref{bandlimit:exp*}), we 
define $\exp^*_1\left[f\right]$ for $f\in \Lp{1}$ by the formula
\begin{equation}
\exp_1^*\left[f\right](\xi) = 
f(\xi) + \frac{f*f(\xi)}{2!} + \frac{f*f*f(\xi)}{3!} + \cdots,
\label{bandlimit:exp1}
\end{equation}
and we define $\exp^*_2\left[f\right]$ for $f\in \Lp{1}$ via the formula
\begin{equation}
\exp_2^*\left[f\right](\xi) = 
\frac{f*f(\xi)}{2!} + \frac{f*f*f(\xi)}{3!} + \cdots.
\label{bandlimit:exp2}
\end{equation}
That is, $\exp_1^*\left[f\right]$ is obtained by truncating the leading
term of $\exp^*\left[f\right]$  and $\exp_2^*\left[f\right]$ 
is obtained by truncated the first two leading terms of 
$\exp^*\left[f\right]$.

Finally, we define functions $\psi(\xi)$ and $v(\xi)$ using the formulas
\begin{equation}
\psi(\xi) = \widehat{\sigma}(\xi)
\end{equation}
and
\begin{equation}
v(\xi) = \widehat{p}(\xi).
\label{bandlimit:def_of_v}
\end{equation}

Applying the Fourier transform to both sides of (\ref{bandlimit:integral_equation})
results in the nonlinear equation
\begin{equation}
\psi(\xi) =
R\left[\psi\right](\xi),
\label{bandlimit:equation}
\end{equation}
where  $R\left[f\right]$ is defined for $f \in \Lp{1}$ by the formula
\begin{equation}
R\left[f\right](\xi) = 
\frac{1}{4}\OpWTb{f}*\OpWTb{f}(\xi)
- 4 \lambda^2 
\exp_2^*\left[\OpWb{f}\hskip .2em\right](\xi)
+ v(\xi).
\label{bandlimit:definition_of_R}
\end{equation}
%

\label{section:bandlimit}
\end{section}

\begin{section}{Existence of solutions of the band-limited equation.}

In this section, we give conditions
under which the sequence $\{\psi_n\}_{n=0}^\infty$ of fixed point iterates 
for (\ref{bandlimit:equation}) obtained by using the function $v$ 
defined by (\ref{bandlimit:def_of_v}) as an initial approximation  converges.
More explicitly,  $\psi_0$ is defined by the formula
\begin{equation}
\psi_0(\xi) = v(\xi),
\label{convergence:psi0}
\end{equation}
and for each integer $n \geq 0$, $\psi_{n+1}$  is obtained from $\psi_n$ via 
\begin{equation}
\psi_{n+1}(\xi) = 
R\left[\psi_{n}\right](\xi).
\label{convergence:psin}
\end{equation}
%

%
\vskip 1em
\begin{theorem}
Suppose that $\lambda > 0$ is a real number,  and that $v \in \Lp{1}$ such that
\begin{equation}
\|v\|_1 \leq \frac{\lambda^2}{18}.
\label{convergence:theorem1:assumption}
\end{equation}

Theen the sequence  $\{\psi_n\}_{n=0}^\infty$  defined by (\ref{convergence:psi0})
and (\ref{convergence:psin}) converges in $\Lp{1}$ norm to a function $\psi$  which satisfies
the equation (\ref{bandlimit:equation}) for almost all $\xi \in \mathbb{R}$.
\label{convergence:theorem1}
\end{theorem}
%
%
\begin{proof}
We observe that the Fr\'echet derivative 
(see Section~\ref{section:preliminaries:frechet})
of $R$ at  $f$ is the linear operator  $R_f':\Lp{1}\to\Lp{1}$ given by the formula
\begin{equation}
R'_f\left[h\right](\xi) = 
\frac{\OpWTb{f}*\OpWTb{h}(\xi)}{2}
- 4 \lambda^2 
\exp_1^*\left[\OpWb{f}\right]*\OpWb{h}
(\xi).
\label{convergence:definition_of_R'}
\end{equation}
From formulas~(\ref{bandlimit:definition_of_W}) and (\ref{bandlimit:definition_of_W'})
and the definition of $b(\xi)$ we see that
\begin{equation}
\left\|\OpWb{f}\right\|_1 \leq \frac{\|f\|_1}{2\lambda^2}
\label{convergence:Sbounds1}
\end{equation}
and
\begin{equation}
\left\|\OpWTb{f}\right\|_1 \leq \frac{\|f\|_1}{\sqrt{2}\lambda}
\label{convergence:Sbounds2}
\end{equation}
for all $f \in \Lp{1}$.  From (\ref{convergence:definition_of_R'}),
(\ref{convergence:Sbounds1}) and (\ref{convergence:Sbounds2})
we conclude that
\begin{equation}
\begin{aligned}
\left\|R'_f\left[h\right]\right\|_1
&\leq 
\frac{1}{2}
\left\|\OpWTb{f}\right\|_1
\left\|\OpWTb{h}\right\|_1
+ 4 \lambda^2 
\left\|\OpWb{f}\right\|_1
\exp\left(\left\|\OpWb{f}\right\|_1\right)
\left\|\OpWb{h}\right|_1
\\
&\leq 
\frac{\|f\|_1 \|h\|_1}{4\lambda^2}
+ 4 \lambda^2 
\frac{\|f\|_1}{2\lambda^2}
\frac{\|h\|_1}{2\lambda^2}
\exp\left(\frac{\|f\|_1}{2\lambda^2}\right)
\\
&\leq 
\left(
\frac{\|f\|_1}{4\lambda^2}
+
\frac{\|f\|_1}{\lambda^2}
\exp\left(\frac{\|f\|_1}{2\lambda^2}\right)
\right)\|h\|_1
\end{aligned}
\label{convergence:bound_R'}
\end{equation}
for all $f$ and $h$ in $\Lp{1}$.  Similarly, by combining 
(\ref{bandlimit:definition_of_R}), 
 (\ref{convergence:Sbounds1}) and (\ref{convergence:Sbounds2})
we conclude that
\begin{equation}
\begin{aligned}
\left\|\OpR{f}\right\|_1
&\leq 
\frac{1}{4}
\left\|\OpWTb{f}\right\|_1^2
+ 4 \lambda^2 
\left\|\OpWb{f}\right\|_1^2
\exp\left(\left\|\OpWb{f}\right\|_1\right)
+ \|v\|_1
\\
&\leq 
\frac{\|f\|_1^2}{8\lambda^2}
+ 
 \frac{\|f\|_1^2}{\lambda^2}
\exp\left(\frac{\|f\|_1}{2\lambda^2}\right)
+ \|v\|_1
\end{aligned}
\label{convergence:bound_R}
\end{equation}
whenever $f \in \Lp{1}$.

Now we set $r = \lambda^2/9$ and let $B$ denote the
closed ball  of radius $r$ centered at $0$ in $\Lp{1}$.
Since
\begin{equation}
\frac{1}{9} < W_0\left(\frac{1}{8}\right),
\label{convergence:w0in}
\end{equation}
where $W_0$ denotes the branch of the Lambert $W$ function
which is real-valued and greater than or equal to $-1$
on the interval $[-1/e,\infty)$ (see Section~\ref{section:preliminaries:LambertW}), 
we conclude that
\begin{equation}
\frac{r}{\lambda^2}\exp\left(\frac{r}{2\lambda^2}\right) < \frac{1}{4}.
\label{convergence:rinequality1}
\end{equation}
According to Theorem~\ref{preliminaries:LambertW:theorem3} of 
Section~\ref{section:preliminaries:LambertW}
and (\ref{convergence:w0in}), it also the case that
\begin{equation}
\frac{r}{\lambda^2} < \frac{1}{8}.
\label{convergence:rinequality2}
\end{equation}
We combine (\ref{convergence:rinequality1}),
(\ref{convergence:rinequality2})
 with (\ref{convergence:bound_R'}) to  conclude that
\begin{equation}
\left\|R'_f\left[h\right]\right\|_1
\leq 
\left(
\frac{r}{4\lambda^2}  + \frac{r}{\lambda^2} \exp\left(\frac{r}{2\lambda^2}\right)
\right) \|h\|_1
< \left(\frac{1}{32} + \frac{1}{4}\right) \|h\|_1 < \frac{1}{2} \|h\|_1
\label{convergence:contraction}
\end{equation}
for all $h \in \Lp{1}$ and $f \in B$.  In other words, the  $\Lp{1} \to \Lp{1}$ operator
norm of the linear operator $R_f'$ is bounded by $1/2$ whenever $f$ is in the ball $B$.

Similarly, we insert (\ref{convergence:rinequality1}), 
(\ref{convergence:rinequality2})  and  (\ref{convergence:theorem1:assumption})
into (\ref{convergence:bound_R})  to conclude that
\begin{equation}
\begin{aligned}
\left\|\OpR{f}\right\|_1
&\leq 
\frac{r^2}{8\lambda^2}  + \frac{r^2}{\lambda^2} \exp\left(\frac{r}{2\lambda^2}\right)
+ \frac{r}{2} \\
&\leq
r
\left(
\frac{r}{8\lambda^2}  + \frac{r}{\lambda^2} \exp\left(\frac{r}{2\lambda^2}\right)
+ \frac{1}{2} 
\right)\\
&\leq 
r
\left(
\frac{1}{64}  + \frac{1}{4} + \frac{1}{2} 
\right)\\
&\leq r
\end{aligned}
\label{convergence:selfmap}
\end{equation}
for all $f \in B$.

Together with Theorem \ref{preliminaries:frechet:mvt}
of Section~\ref{section:preliminaries:frechet},
formula (\ref{convergence:contraction}) implies that the operator $R$ is a contraction
on the ball $B$ while (\ref{convergence:selfmap}) says that it maps the ball $B$ into itself.
We now apply the contraction mapping theorem 
(Theorem~\ref{preliminaries:frechet:cmp}
in Section~\ref{section:preliminaries:frechet})
to conclude that 
any sequence of fixed point iterates  for (\ref{bandlimit:equation})
which originates 
in $B$ will converge to a solution of (\ref{bandlimit:equation}).
Since $\{\psi_n\}$ is such a sequence, we are done.
\end{proof}


%
%
If $\psi$ is  a solution of (\ref{bandlimit:equation})
then the function $\sigma$ defined by the formula
\begin{equation}
\sigma(x) = \int_{-\infty}^{\infty} \exp(ix \xi) \psi(\xi)\ d\xi
\label{convergence:sigma}
\end{equation}
is clearly a solution of the band-limited integral equation 
(\ref{bandlimit:integral_equation}).  Note that because $\psi \in \Lp{1}$,
 the integral in (\ref{convergence:sigma}) is well-defined and
$\sigma$ is an element of the space $\Czero$ of continuous
functions which vanish at infinity.  
We record this observation as follows:

\vskip 1em
\begin{theorem}
Suppose that $\lambda >0$ is a real number.  Suppose also that
$p \in \Lp{1}$ such that $\widehat{p}\in \Lp{1}$ and 
\begin{equation}
\left\|\widehat{p}\right\|_1 < \frac{\lambda^2}{18}.
\end{equation}
Then there exists a function $\sigma \in \Czero$ which is a solution
of the integral equation (\ref{bandlimit:integral_equation}).  
\label{convergence:theorem2}
\end{theorem}

\vskip 1em
\begin{remark}
Since $\sigma$ is not necessarily in $\Lp{1}$, the integral
\begin{equation}
\int_{-\infty}^{\infty} \exp(-ix\xi) \sigma(x)\ dx
\end{equation}
need not exist.  Nor is the existence of the improper integral
\begin{equation}
\lim_{R\to\infty} \int_{-R}^{R} \exp(-ix\xi) \sigma(x)\ dx
\label{bandlimit:l2fourier}
\end{equation}
guaranteed.  However, when viewed as a tempered distribution,
the Fourier transform of $\sigma$ exists and is $\psi$; that is to say,
\begin{equation}
\int_{-\infty}^{\infty} \psi(x) f(x)\ dx
=
\int_{-\infty}^{\infty} \sigma(x) \widehat{f}(x)\ dx
\end{equation}
for all functions $f \in \Sr$.

In the next section we will prove that under 
additional assumptions on $v$, $\psi$  lies in $\Lp{2}$.  This implies that
$\sigma \in \Lp{2}$ and ensures the convergence of the
improper Riemann integrals (\ref{bandlimit:l2fourier}).
\end{remark}

\label{section:bandlimit:solution}
\end{section}

\begin{section}{Fourier estimate}

In this section, we derive a pointwise estimate on the solution
$\psi$ of Equation~(\ref{bandlimit:equation}) 
under additional assumptions on the function $v$.

%
\vskip 1em
\begin{lemma}
Suppose that $a$ and $C$ are real numbers such that
\begin{equation}
0 \leq C < a.
\label{spectrum:S2:assumption1}
\end{equation}
Suppose also that  $f \in \Lp{1}$, and that
\begin{equation}
\left|f(\xi)\right| \leq C \exp(-a |\xi|)
\ \ \ \mbox{for all}\ \ \xi\in\mathbb{R}.
\label{spectrum:S2:assumption2}
\end{equation}
Then
\begin{equation}
\left|\exp_2^*\left[f\right](\xi)\right|
\leq 
\frac{C^2}{2\pi} \exp(-a|\xi|)
\frac{1+a|\xi|}{a}
\exp\left(\frac{C}{2\pi a}\right)
\exp
\left(
\frac{C}{2 \pi }|\xi|
\right)
\ \ \ \mbox{for all}\ \ \ \xi\in \mathbb{R},
\label{spectrum:S2:bound}
\end{equation}
where $\exp_2^*$ is the operator defined in (\ref{bandlimit:exp2}).
\label{spectrum:S2}
\end{lemma}
\begin{proof}
Let 
\begin{equation}
g(\xi) = C \exp(-a|\xi|)
\end{equation}
and for each integer $m > 0$, denote by $g_m$ the
$m$-fold convolution product of the function $g$.
That is to say that $g_1$ is defined via the formula
\begin{equation}
g_1(\xi) = g(\xi)
\end{equation}
and for each integer $m > 0$, $g_{m+1}$ is defined in terms
of $g_m$ by the formula
\begin{equation}
g_{m+1}(\xi) = g_{m}*g(\xi).
\end{equation}
We observe that for each integer $m > 0$ and all $\xi \in \mathbb{R}$,
\begin{equation}
g_m(\xi)
=
2 \sqrt{aC} \left(\frac{C|\xi|}{2\pi}\right)^{m-1/2}
\frac{K_{m-1/2}(a|\xi|)}{\Gamma(m)},
\label{spectrum:S2:convolution_product}
\end{equation}
where $K_\nu$ denotes the modified Bessel function of the second kind of order $\nu$
(see Section~\ref{section:preliminaries:bessel}).
By repeatedly applying Theorem~\ref{preliminaries:bessel:theorem1} 
of Section~\ref{section:preliminaries:bessel}, we conclude that
for all integers $m >0$ and all real $t$,
\begin{equation}
\begin{aligned}
K_{m-1/2}(t) 
&\leq K_{1/2}(t) 
\prod_{j=1}^{m-1} \left(\frac{2\left(j-\frac{1}{2}\right)}{t}+1 \right)\\
&=
K_{1/2}(t)
\left(\frac{2}{t}\right)^{m-1}
\frac{\Gamma\left(\frac{1+t}{2}+m-1\right)}{\Gamma\left(\frac{1+t}{2}\right)}.
\end{aligned}
\label{spectrum:S2:kbound1}
\end{equation}
We insert the identity
\begin{equation}
K_{1/2}(t) = \sqrt{\frac{\pi}{2t}} \exp(-t)
\end{equation}
into (\ref{spectrum:S2:convolution_product}) in order to conclude that
for all integers $m > 0$ and all real numbers $t > 0$,
\begin{equation}
\begin{aligned}
K_{m-1/2}(t) 
&\leq
\frac{\sqrt{\pi}}{2}  \left(\frac{t}{2}\right)^{1/2-m} \exp(-t) 
\frac{
\Gamma\left( \frac{1+t}{2}  + m-1\right)}
{\Gamma\left( \frac{1+t}{2}\right)}.
\end{aligned}
\label{spectrum:S2:kbound}
\end{equation}
By combining (\ref{spectrum:S2:kbound}) and (\ref{spectrum:S2:convolution_product})  we conclude
that
\begin{equation}
g_m(\xi)
\leq 
C \exp(-a|\xi|)\left(\frac{C}{\pi a}\right)^{m-1} 
\frac{\Gamma\left(\frac{1+a|\xi|}{2}+m-1\right)}
{\Gamma(m)\Gamma\left(\frac{1+a|\xi|}{2}\right)}
\label{spectrum:S2:gmbound}
\end{equation}
for all integers $m > 0$ and all $\xi \neq 0$.  
Moreover, the limit
as $\xi \to 0$ of each side of (\ref{spectrum:S2:gmbound}) is finite
and the two limits are equal, so (\ref{spectrum:S2:gmbound}) in fact holds 
for all $\xi \in \mathbb{R}$.  We sum (\ref{spectrum:S2:gmbound}) over $m=2,3,\ldots$ 
in order to conclude that
\begin{equation}
\begin{aligned}
\exp_2^*\left[g\right](\xi)
&\leq 
C \exp(-a|\xi|)
\sum_{m=2}^\infty
\left(\frac{C}{\pi a}\right)^{m-1}
\frac{\Gamma\left(\frac{1+a|\xi|}{2}+m-1 \right)}
{\Gamma(m+1)\Gamma(m)\Gamma\left(\frac{1+a|\xi|}{2}\right)}
\\
&=
C \exp(-a|\xi|)
\sum_{m=1}^\infty
\left(\frac{C}{\pi a}\right)^{m}
\frac{\Gamma\left(\frac{1+a|\xi|}{2}+m \right)}
{\Gamma(m+2)\Gamma(m+1)\Gamma\left(\frac{1+a|\xi|}{2}\right)}
\end{aligned}
\label{spectrum:S2:interbound}
\end{equation}
for all $\xi \in \mathbb{R}$.     Now we observe that
\begin{equation}
\frac{1}{\Gamma(m+2)} \leq  \left(\frac{1}{2}\right)^{m}
\ \ \ \mbox{for all}\ \ m=0,1,2,\ldots.
\label{spectrum:S2:gamma_bound}
\end{equation}
Inserting (\ref{spectrum:S2:gamma_bound}) into (\ref{spectrum:S2:interbound})
yields
\begin{equation}
\begin{aligned}
\exp_2^*\left[g\right](\xi)
&\leq
C \exp(-a|\xi|)
\sum_{m=1}^\infty 
\left(\frac{C}{2 \pi a}\right)^{m}
\frac{\Gamma\left(\frac{1+a|\xi|}{2}+m\right)}
{\Gamma(m+1)\Gamma\left(\frac{1+a|\xi|}{2}\right)}
\end{aligned}
\label{spectrum:S2:binomial1}
\end{equation}
for all $\xi \in \mathbb{R}$.    Now we apply the binomial theorem
(Theorem~\ref{preliminaries:binomial_theorem}
 of Section~\ref{section:preliminaries:binomial}),
which is justified since $C < a$, to conclude that
\begin{equation}
\begin{aligned}
\exp_2^*\left[g\right](\xi)
&\leq
C \exp(-a|\xi|)
\left(
\left(1-\frac{C}{2\pi a}\right)^{-\frac{1+a|\xi|}{2}}-1
\right)
\\
&=
C \exp(-a|\xi|)
\left(
\exp
\left(
\frac{1+a|\xi|}{2}
\log\left(\frac{1}{1-\frac{C}{2\pi a}}\right)
\right)-1
\right)
\end{aligned}
\label{spectrum:S2:binomial2}
\end{equation}
for all $\xi \in \mathbb{R}$.  We observe that
\begin{equation}
\exp(x) -1 \leq x \exp(x)
\ \ \ \mbox{for all}\ \ x \geq 0,
\label{spectrum:S2:1}
\end{equation}
and
\begin{equation}
1 \leq \log\left(\frac{1}{1-x}\right) \leq 2x 
\ \ \ \mbox{for all}\ \ 0 \leq x \leq \frac{1}{2\pi}.
\label{spectrum:S2:2}
\end{equation}
By combining (\ref{spectrum:S2:1}) and (\ref{spectrum:S2:2})
with (\ref{spectrum:S2:binomial2}) we conclude that
\begin{equation}
\begin{aligned}
\exp_2^*\left[g\right](\xi)
&\leq
C \exp(-a|\xi|)
\frac{1+a|\xi|}{2}
\log\left(\frac{1}{1-\frac{C}{2\pi a}}\right)
\exp
\left(
\frac{1+a|\xi|}{2}
\log\left(\frac{1}{1-\frac{C}{2\pi a}}\right)
\right)
\\
&\leq
\frac{C^2}{2\pi} \exp(-a|\xi|)
\frac{1+a|\xi|}{a}
\exp\left(\frac{C}{2\pi a}\right)
\exp
\left(
\frac{C}{2 \pi }|\xi|
\right)
\\
\end{aligned}
\label{spectrum:S2:binomial3}
\end{equation}
for all $\xi \in\mathbb{R}$.   Owing to (\ref{spectrum:S2:assumption2}),
\begin{equation}
\left|\exp_2^*\left[f\right](\xi)\right| \leq 
\exp_2^*\left[g\right](\xi)
\ \ \ \mbox{for all} \ \ \xi\in\mathbb{R}.
\end{equation}
By combining this observation with (\ref{spectrum:S2:binomial3}),
we obtain (\ref{spectrum:S2:bound}), which completes the proof.
\end{proof}
\begin{remark}
Kummer's confluent hypergeometric function $M(a,b,z)$ is defined 
by the series
\begin{equation}
M(a,b,z) = 1 + \frac{az}{b} + 
\frac{(a)_2 z^2}{(b)_2 2!} + 
 \frac{(a)_3 z^3}{(b)_3 3!} + \cdots,
\label{kummer:confluent}
\end{equation}
where $(a)_n$ is the Pochhammer symbol
\begin{equation}
(a)_n = \frac{\Gamma(a+n)}{\Gamma(a)} = a (a+1) (a+2) \ldots (a+n-1).
\end{equation}
By comparing the definition of $M(a,b,z)$ with (\ref{spectrum:S2:interbound}),
we conclude that
\begin{equation}
\left|\exp_2^*\left[f\right](\xi)\right|
\leq C \exp(-a|\xi|)\left(M\left(\frac{1+a|\xi|}{2},2,\frac{C}{\pi a} \right)-1\right)
\ \ \ \mbox{for all} \ \ \xi\in\mathbb{R}
\label{spectrum:tighter_bound}
\end{equation}
provided
\begin{equation}
\left|f(\xi)\right|\leq C \exp(-a|\xi|)
\ \ \ \mbox{for all} \ \ \xi\in\mathbb{R}.
\end{equation}
The weaker bound (\ref{spectrum:S2:bound}) is sufficient for 
our immediate purposes, but formula (\ref{spectrum:tighter_bound}) might serve
as a basis for improved estimates on solutions of Kummer's equation.
\end{remark}

The following lemma is a special case of Formula~(\ref{spectrum:S2:convolution_product}).

%
\vskip 1em
\begin{lemma}
Suppose that $C \geq 0$ and $a >0 $ are real numbers, and that
$f \in \Lp{1}$ such that
\begin{equation}
\left|f(\xi)\right| \leq C \exp\left(-a|\xi|\right)
\ \ \ \mbox{for all}\ \ \xi \in \mathbb{R}.
\label{spectrum:S1:assumption}
\end{equation}
Then
\begin{equation}
\left|f*f(\xi)\right|
\leq
C^2 
\exp(-a|\xi|)
\left(\frac{1+a|\xi|}{a}\right)
\ \ \ \mbox{for all} \ \ \xi \in \mathbb{R}.
\end{equation}
\label{spectrum:S1}
\end{lemma}

We will also make use of the following elementary observation.
\vskip 1em
\begin{lemma}
Suppose that $a > 0$ is a real number.  Then
\begin{equation}
\exp(-a|\xi|)|\xi| \leq \frac{1}{a \exp(1)} \ \ \ \mbox{for all} \ \ \xi\in\mathbb{R}.
\end{equation}
\label{spectrum:exp_lemma}
\end{lemma}

We combine Lemmas~\ref{spectrum:S2} and  \ref{spectrum:S1} 
with (\ref{convergence:Sbounds1}) and (\ref{convergence:Sbounds2})
in order to obtain the following key estimate.

%
%
\vskip 1em
\begin{theorem}
Suppose that $\Gamma > 0$, $\lambda >0$, $a > 0$ and $C \geq 0 $ are real numbers such that
\begin{equation}
0 \leq C < 2 a\lambda^2.
\label{spectrum:STheorem:assumption1}
\end{equation}
Suppose also that  $f \in \Lp{1}$ such that
\begin{equation}
\left|f(\xi)\right| \leq C \exp(-a|\xi|)
\ \ \ \mbox{for all}\ \ |\xi| \leq \sqrt{2}\lambda,
\label{spectrum:STheorem:assumption2}
\end{equation}
and that $v \in \Lp{1}$ such that
\begin{equation}
\left|v(\xi)\right| \leq \Gamma \exp(-a|\xi|)
\ \ \ \mbox{for all} \ \ \xi\in\mathbb{R}.
\label{spectrum:STheorem:vassumption}
\end{equation}
Suppose further  that $R$ is the operator defined via (\ref{bandlimit:definition_of_R}).
Then
\begin{equation}
\left|\OpR{f}(\xi)\right|
\leq
\exp(-a|\xi|)
\left(
\frac{C^2}{\lambda^2}  \left(\frac{1+a|\xi|}{a}\right)
\left(
\frac{1}{8} + \frac{1}{2\pi}\exp\left(\frac{C}{4\pi \lambda^2a}\right)
\exp\left(\frac{C}{4\pi \lambda^2}|\xi|\right)
\right)
+\Gamma
\right)
\label{spectrum:STheorem:bound}
\end{equation}
for all $\xi \in \mathbb{R}$.
\label{spectrum:STheorem}
\end{theorem}

\begin{proof}
We define the operator $R_1$ via the formula
\begin{equation}
R_1\left[f\right](\xi) = \frac{1}{4}
\OpWTb{f}* \OpWTb{f}(\xi)
\end{equation}
and  $R_2$ by the formula
\begin{equation}
R_2\left[f\right](\xi) = -4\lambda^2\exp_2^*\left[\OpWb{f}\right](\xi),
\end{equation}
where $W_b$ and $\widetilde{W}_b$ are defined as in 
Section~\ref{section:bandlimit}.  Then
\begin{equation}
\OpR{f}(\xi) = R_1\left[f\right](\xi) + R_2\left[f\right](\xi) + v(\xi)
\end{equation}
for all $\xi \in \mathbb{R}$.  We observe that
\begin{equation}
\left|\OpWTb{f}(\xi)\right| \leq \frac{C}{\sqrt{2}\lambda} \exp(-a|\xi|)
\ \ \ \mbox{for all}\ \ \xi\in\mathbb{R}.
\label{spectrum:STheorem:1}
\end{equation}
By combining Lemma~\ref{spectrum:S1} with (\ref{spectrum:STheorem:1})
we obtain
\begin{equation}
\left|R_1\left[f\right](\xi)\right|
\leq 
\frac{C^2}{8\lambda^2} \exp(-a|\xi|)\left(\frac{1+a|\xi|}{a}\right)
\ \ \ \mbox{for all}\ \ \xi\in\mathbb{R}.
\label{spectrum:STheorem:2}
\end{equation}
Now we observe that
\begin{equation}
\left|\OpWb{f}(\xi)\right| \leq \frac{C}{2\lambda^2} \exp(-a|\xi|)
\ \ \ \mbox{for all}\ \ \xi\in\mathbb{R}.
\label{spectrum:STheorem:3}
\end{equation}
Combining Lemma~\ref{spectrum:S2} with (\ref{spectrum:STheorem:3})
yields  
\begin{equation}
\begin{aligned}
\left|R_2\left[f\right](\xi)\right|
&\leq 
\frac{C^2}{2\pi\lambda^2} \exp(-a|\xi|)
\left(\frac{1+a|\xi|}{a}\right)
\exp\left(\frac{C}{4 \pi \lambda^2 a}\right)
\exp
\left(
\frac{C}{4\pi \lambda^2 }|\xi|
\right)
\end{aligned}
\label{spectrum:STheorem:4}
\end{equation}
for all $\xi \in \mathbb{R}$.  Note that (\ref{spectrum:STheorem:assumption1})
ensures that the hypothesis (\ref{spectrum:S2:assumption1}) in Lemma~\ref{spectrum:S2} is satisfied.
We combine (\ref{spectrum:STheorem:2})
with (\ref{spectrum:STheorem:4}) and (\ref{spectrum:STheorem:vassumption}) 
in order to obtain (\ref{spectrum:STheorem:bound}), and by so doing
we complete the proof.
\end{proof}

\begin{remark}
Note that  Theorem~\ref{spectrum:STheorem} only
requires that $f(\xi)$ satisfy a bound on the interval $[-\sqrt{2}\lambda,\sqrt{2}\lambda]$ 
and not on the entire real line.  
\end{remark}

%
%

In the next theorem,
we use Theorem~\ref{spectrum:STheorem} to  bound the solution of (\ref{bandlimit:equation}) 
under an assumption on the decay of $v$.

\vskip 1em
\begin{theorem}
Suppose that $\lambda > 0$, $a > 0$ 
and $\Gamma \geq 0$ are real numbers such that
\begin{equation}
\lambda \geq 6 \max\left\{\Gamma,\frac{1}{a}\right\}.
\label{spectrum:theorem2:assumption1}
\end{equation}
Suppose also that $v \in \Lp{1}$ such that
\begin{equation}
\left|v(\xi)\right| \leq  \Gamma \exp(-a|\xi|)
\ \ \ \mbox{for all} \ \ \xi\in\mathbb{R}.
\label{spectrum:theorem2:assumption2}
\end{equation}
Then there exists a solution of $\psi(\xi)$ of equation
(\ref{bandlimit:equation})  such that
\begin{equation}
\left|\psi(\xi)\right| \leq 
2\Gamma
\exp\left(-\left(a-\frac{1}{\lambda}\right)|\xi|\right) 
\ \ \ \mbox{for all} \ \ \xi\in\mathbb{R}.
\label{spectrum:theorem2:bound}
\end{equation}
\label{spectrum:theorem2}
\end{theorem}

\begin{proof}
Due to (\ref{spectrum:theorem2:assumption1}) and
(\ref{spectrum:theorem2:assumption2}),
\begin{equation}
\|v\|_1 \leq \frac{\lambda^2}{18}.
\label{spectrum:theorem2:v1}
\end{equation}
It follows from Theorem~\ref{convergence:theorem1} and (\ref{spectrum:theorem2:v1})
that  a solution  $\psi(\xi)$ of (\ref{bandlimit:equation}) is obtained
as the limit of the sequence of fixed point iterates $\{\psi_n(\xi)\}$
defined by the formula
\begin{equation}
\psi_0(\xi) = v(\xi)
\end{equation}
and the recurrence
\begin{equation}
\psi_{n+1}(\xi) = \OpR{\psi_{n}}(\xi).
\end{equation}
%
%
We now  derive pointwise estimates on the iterates $\psi_n(\xi)$ 
in order to establish (\ref{spectrum:theorem2:bound}).

Let $\{\beta_k\}_{k=0}^\infty$ be  the sequence of real numbers be generated
by the recurrence relation
\begin{equation}
\beta_{k+1} = \frac{\beta_k^2}{2 \lambda}+\Gamma
\end{equation}
with the initial value
\begin{equation}
\beta_0 = \Gamma.
\end{equation}
It can be established by induction that (\ref{spectrum:theorem2:assumption1})
%
implies that this sequence is bounded above by $2\Gamma$ and monotonically increasing,
and hence $\beta_k$   converges to a real number $\beta$ such that
$0\leq \beta\leq 2\Gamma$.

Now suppose that $n \geq 0$ is an integer, and that
\begin{equation}
\left|\psi_n(\xi)\right|
\leq \beta_n \exp(-a|\xi|)
\ \ \ \mbox{for all}\ \ \ |\xi|\leq\sqrt{2}\lambda.
\label{spectrum:theorem2:bound2}
\end{equation}
When $n=0$, this is simply the assumption (\ref{spectrum:theorem2:assumption2}).
The function $\psi_{n+1}(\xi)$ is obtained from $\psi_n(\xi)$ via the formula
\begin{equation}
\psi_{n+1}(\xi) = \OpR{\psi}(\xi).
\label{spectrum:theorem2:psi}
\end{equation}
We combine Theorem~\ref{spectrum:STheorem} with (\ref{spectrum:theorem2:psi})
and (\ref{spectrum:theorem2:bound2}) to conclude that
\begin{equation}
\begin{aligned}
\left|\psi_{n+1}(\xi)\right|
\leq
\exp(-a|\xi|)
\left(
\frac{\beta_n^2}{\lambda^2}  \left(\frac{1+a|\xi|}{a}\right)
\left(
\frac{1}{8} + \frac{1}{2\pi}\exp\left(\frac{\beta_n}{4\pi \lambda^2 a}\right)
\exp\left(\frac{\beta_n}{4\pi \lambda^2}|\xi|\right)
\right)
+\Gamma
\right)
\end{aligned}
\label{spectrum:theorem2:2}
\end{equation}
for all $\xi \in \mathbb{R}$.    The application of  Theorem~\ref{spectrum:STheorem} 
is justified: the hypothesis (\ref{spectrum:STheorem:assumption1}) 
is satisfied since
\begin{equation}
\beta_n \leq 2\Gamma \leq 2 \lambda^2 a
\end{equation}
for all integers $n \geq 0$.   
We restrict  $\xi$ to the interval $[-\sqrt{2}\lambda,\sqrt{2}\lambda]$ in 
(\ref{spectrum:theorem2:2}) and use the fact that
\begin{equation}
\frac{1}{a \lambda} < \frac{1}{6},
\end{equation}
which is a consequence of (\ref{spectrum:theorem2:assumption1}), in order to conclude that
\begin{equation}
\begin{aligned}
\left|\psi_{n+1}(\xi)\right|
&\leq
\exp(-a|\xi|)
\left(
\frac{\beta_n^2}{\lambda^2}  \left(\frac{1+a\sqrt{2}\lambda}{a}\right)
\left(
\frac{1}{8} + \frac{1}{2\pi}\exp\left(\frac{\beta_n}{4\pi \lambda^2 a}\right)
\exp\left(\frac{\beta_n}{4\pi \lambda^2}\sqrt{2}\lambda\right)
\right)
+\Gamma
\right)
\\
&\leq
\exp(-a|\xi|)
\left(
\frac{\beta_n^2}{\lambda}  
\left(\frac{1}{6} + \sqrt{2}\right)
\left(
\frac{1}{8} + \frac{1}{2\pi}\exp\left(\frac{\beta_n}{24\pi \lambda}\right)
\exp\left(\frac{\beta_n}{2\sqrt{2}\pi \lambda}\right)
\right)
+\Gamma
\right)
\end{aligned}
\label{spectrum:theorem2:3}
\end{equation}
for all $|\xi|\leq \sqrt{2}\lambda$.    Now we combine (\ref{spectrum:theorem2:3})
with the inequality
\begin{equation}
\frac{\beta_n}{\lambda} \leq \frac{2\Gamma}{\lambda} < \frac{1}{3}
\end{equation}
and the observation that
\begin{equation}
\left(\frac{1}{6} + \sqrt{2}\right)
\left(
\frac{1}{8} + \frac{1}{2\pi}\exp\left(\frac{1}{72\pi}\right)
\exp\left(\frac{1}{6\sqrt{2}\pi}\right)
\right)
\leq \frac{1}{2}
\end{equation}
in order to conclude that
\begin{equation}
\begin{aligned}
\left|\psi_{n+1}(\xi)\right|
&\leq
\exp(-a|\xi|)
\left(
\frac{\beta_n^2}{\lambda}  
\left(\frac{1}{6} + \sqrt{2}\right)
\left(
\frac{1}{8} + \frac{1}{2\pi}\exp\left(\frac{1}{72\pi}\right)
\exp\left(\frac{1}{6\sqrt{2}\pi}\right)
\right)
+\Gamma
\right)
\\
&\leq
\left(\frac{\beta_n^2}{2 \lambda}+\Gamma \right)
\exp(-a|\xi|)
\\
&=
\beta_{n+1} \exp(-a|\xi|)
\end{aligned}
\label{spectrum:theorem2:4}
\end{equation}
for  all $|\xi| \leq \sqrt{2}\lambda$.
%
We conclude by induction that  (\ref{spectrum:theorem2:bound2}) holds for all integers
$n \geq 0$.
 
The sequence $\{\psi_n(\xi)\}$ converges to $\psi(\xi)$  in $\Lp{1}$ norm 
(and hence a subsequence of $\psi_n(\xi)$ converges to $\psi(\xi)$
pointwise almost everywhere) and (\ref{spectrum:theorem2:bound2}) holds for all integers $n \geq 0$.
Moreover, for all integers $n \geq 0$, $\beta_{n+1} \leq 2 \Gamma$.
From these observations we conclude that
\begin{equation}
\left|\psi(\xi)\right| \leq 2\Gamma \exp(-a|\xi|)
\label{spectrum:theorem2:psibound1}
\end{equation}
for almost all $|\xi| \leq \sqrt{2}\lambda$.  
We also  observe that $\psi(\xi)$ is a fixed point of the operator $R$, so that
\begin{equation}
\psi(\xi) = \OpR{\psi}(\xi)
\label{spectrum:theorem2:fixed_point}
\end{equation}
for all $\xi\in\mathbb{R}$.
Clearly,
\begin{equation}
R\left[f\right](\xi) = R\left[g\right](\xi)
\end{equation}
for all $\xi\in\mathbb{R}$ 
if $f(\xi) = g(\xi)$ for almost all $|\xi| \leq \sqrt{2}\lambda$,
so  (\ref{spectrum:theorem2:psibound1}) in fact holds for all $|\xi| \leq \sqrt{2}\lambda$.

We now apply Theorem~\ref{spectrum:STheorem} 
to the function  $\psi(\xi)$ (which is justified
since $2\Gamma < 2\lambda^2 a$) to conclude that
\begin{equation}
\begin{aligned}
\left|\psi(\xi)\right|
\leq
\exp(-a|\xi|)
\left(
\frac{4\Gamma^2}{\lambda^2}  \left(\frac{1+a|\xi|}{a}\right)
\left(
\frac{1}{8} + \frac{1}{2\pi}\exp\left(\frac{2\Gamma}{4\pi \lambda^2 a}\right)
\exp\left(\frac{2\Gamma}{4\pi \lambda^2}|\xi|\right)
\right)
+\Gamma
\right)
\end{aligned}
\label{spectrum:theorem2:psibound2}
\end{equation}
for all $\xi \in \mathbb{R}$.  Note the distinction between
(\ref{spectrum:theorem2:bound2}) and (\ref{spectrum:theorem2:psibound2}) is that the former
only holds for all $\xi$ in the interval $[-\sqrt{2}\lambda,\sqrt{2}\lambda]$,
while the later holds for all $\xi$ on the real line.
%
It follows from (\ref{spectrum:theorem2:assumption1}) that
\begin{equation}
\frac{1}{\lambda a} < \frac{1}{6}
\ \ \ \mbox{and}\ \ \
\frac{\Gamma}{\lambda} < \frac{1}{6}.
\end{equation}
We insert these bounds into (\ref{spectrum:theorem2:psibound2})
in order to conclude that
\begin{equation}
\begin{aligned}
\left|\psi(\xi)\right|
&\leq
\Gamma
\exp(-a|\xi|)
\left(
\frac{2}{3\lambda}  \left(\frac{1+a|\xi|}{a}\right)
\left(
\frac{1}{8} + \frac{1}{2\pi}\exp\left(\frac{1}{72\pi }\right)
\exp\left(\frac{1}{12\pi \lambda}|\xi|\right)
\right)
+1
\right)
\end{aligned}
\label{spectrum:theorem2:psibound3}
\end{equation}
for all $\xi \in \mathbb{R}$.
Now we observe that
\begin{equation}
\frac{1}{2\pi}\exp\left(\frac{1}{72\pi }\right) < \frac{1}{6},
\end{equation}
which, when combined with (\ref{spectrum:theorem2:psibound3}), yields
\begin{equation}
\begin{aligned}
\left|\psi(\xi)\right|
&\leq
\Gamma
\exp(-a|\xi|)
\left(
\left(\frac{1}{9} + \frac{2|\xi|}{3\lambda}\right)
\left(
\frac{1}{8} + 
\frac{1}{6}
\exp\left(\frac{1}{12\pi \lambda}|\xi|\right)
\right)
+1
\right)
\ \ \ \mbox{for all} \ \ \xi\in\mathbb{R}.
\end{aligned}
\label{spectrum:theorem2:psibound4}
\end{equation}
By rearranging the right-hand side of  (\ref{spectrum:theorem2:psibound4}) as
\begin{equation*}
\begin{aligned}
&\Gamma
\exp(-a|\xi|)
\left(
\frac{1}{72} +
\frac{1}{56} \exp\left(\frac{1}{12\pi \lambda}|\xi|\right) +
\frac{|\xi|}{12\lambda} +
\frac{|\xi|}{9\lambda} \exp\left(\frac{1}{12\pi\lambda}|\xi|\right)
+1
\right)
\\
=\ 
&\Gamma
\exp\left(-\left(a-\frac{1}{\lambda}\right)|\xi|\right)
\exp\left(-\frac{1}{\lambda}|\xi|\right)\cdot
\\
&\left(
\frac{1}{72} +
\frac{1}{56} \exp\left(\frac{1}{12\pi \lambda}|\xi|\right) +
\frac{|\xi|}{12\lambda} +
\frac{|\xi|}{9\lambda} \exp\left(\frac{1}{12\pi\lambda}|\xi|\right)
+1
\right)
\end{aligned}
\end{equation*}
and applying Lemma~\ref{spectrum:exp_lemma}, we arrive at
the inequality
\begin{equation}
\left|\psi(\xi)\right|
\leq
\Gamma
\exp\left(-\left(a-\frac{1}{\lambda}\right)|\xi|\right) 
\left(
\frac{1}{72} + 
\frac{1}{56} + 
\frac{1}{12 \exp(1)}  +
 \frac{4}{33 \exp(1)}
+1
\right)
\ \ \ \mbox{for all}\ \ \ \xi\in\mathbb{R},
\label{spectrum:theorem2:psibound5}
\end{equation}
from which (\ref{spectrum:theorem2:bound}) follows immediately.
\end{proof}


Suppose that $\psi \in \Lp{1}$ is a solution of (\ref{bandlimit:equation}).
Then the function $\sigma$ defined by the formula
\begin{equation}
\sigma(x) = \int_{-\infty}^{\infty} \exp(ix \xi) \psi(\xi)\ d\xi
\label{spectrum:sigma}
\end{equation}
is a solution of the integral equation  (\ref{bandlimit:integral_equation}).
However, the Fourier transform of (\ref{spectrum:sigma}) might only
be defined in the sense of tempered distributions and not 
as a Lebesgue or improper Riemann integral.
If, however, we assume the function $p$ appearing in
(\ref{bandlimit:equation}) is an element of $\Lp{1}$ and
impose the hypotheses of Theorem~\ref{spectrum:theorem3} on the 
Fourier transform of $p$, then  $\psi \in \Lp{2}$, from which we conclude
that $\sigma$ is also an element of $\Lp{2}$.  In this event,
there is no difficulty in defining the Fourier transform of $\sigma$.
We record these observations in the following theorem.

\vskip 1em
\begin{theorem}
Suppose that there exist real numbers $\lambda > 0$, $\Gamma > 0$ and $a > 0$ such that
\begin{equation}
\lambda >  6\max\left\{\Gamma, \frac{1}{a}\right\}.
\end{equation}
Suppose also that $p \in \Lp{1}$ such that 
and 
\begin{equation}
\left|\widehat{p}(\xi)\right| \leq 
\Gamma \exp\left(-a|\xi|\right)
\ \ \ \mbox{for all} \ \ \xi \in \mathbb{R}.
\end{equation}
Then there exists a solution  $\sigma(x)$ 
of the integral equation (\ref{bandlimit:integral_equation})
such that
\begin{equation}
\left|\widehat{\sigma}(\xi)\right|
\leq 2 \Gamma \exp\left(-\left(a - \frac{1}{\lambda}\right)|\xi|\right)
\ \ \ \mbox{for all} \ \ \xi \in \mathbb{R}
\label{spectrum:sigmahat_bound}
\end{equation}
\label{spectrum:theorem3}
\end{theorem}

\begin{remark}
The bound (\ref{spectrum:sigmahat_bound}) implies that $\widehat{\sigma}$
decays faster than any polynomial, from which we conclude that 
$\sigma$ is infinitely differentiable.
\end{remark}

\label{section:spectrum}
\end{section}

 \begin{section}{Approximate solution of the original equation}

We would like to insert the solution $\sigma$ of (\ref{bandlimit:integral_equation})
into the original equation (\ref{inteq:integral_equation}).
However, we have no guarantee that $\sigma$ is in $\Lp{1}$, nor do we expect
that $\widehat{\sigma}(\pm 2\lambda) = 0.$
As a consequence,  the integrals defining $\OpT{\sigma}$ might not exist.

To remedy this problem, we define a ``band-limited'' version $\sigma_b$ of $\sigma$ 
by the formula
\begin{equation}
\widehat{\sigma_b}(\xi) = \widehat{\sigma}(\xi) b(\xi),
\end{equation}
where $b(\xi)$ is the function used to define the operator
$T_b$.   We observe that there is no difficulty in applying
$T$ to $\sigma_b$ since
\begin{equation}
\widehat{\sigma_b}(\pm 2\lambda) =0.
\end{equation}
Moreover, $\OpTb{\sigma} = \OpT{\sigma_b}$, so that
\begin{equation}
\sigma(x) = \OpS{\OpT{\sigma_b}}(x) + p(x)
\ \ \ \mbox{for all}\ \ \ x\in\mathbb{R}.
\label{solution:1}
\end{equation}
Rearranging (\ref{solution:1}), we obtain
\begin{equation}
\sigma_b(x) = \OpS{\OpT{\sigma_b}}(x) + 
p_b(x)
\ \ \ \mbox{for all}\ \ \ x\in\mathbb{R},
\end{equation}
where $p_b(x)$ is defined the formula
\begin{equation}
p_b(x) = p(x) + \sigma_b(x) - \sigma(x).
\label{solution:2}
\end{equation}
%




Using (\ref{spectrum:sigmahat_bound}) and (\ref{solution:2}),
we conclude that under the hypotheses of Theorem~\ref{spectrum:theorem3},
\begin{equation}
\begin{aligned}
\|p-p_b\|_\infty &
\leq 
\left\|\widehat{\sigma_b}-\widehat{\sigma}\right\|_1\\
&\leq 
2 \Gamma \int_{|\xi| \geq \lambda} 
\exp\left(-\left(a-\frac{1}{\lambda}\right) |\xi|\right)
\ d\xi\\
&\leq 
\frac{4\Gamma}{a-\frac{1}{\lambda}}\exp\left(-\left(a-\frac{1}{\lambda}\right)\lambda\right)
\\
&\leq
  \frac{24 \Gamma}{5a}\exp\left(-\frac{5}{6}a\lambda\right).
\end{aligned}
\label{solution:3}
\end{equation}
%

%

Together Theorem~\ref{spectrum:theorem3}
and (\ref{solution:3}) imply Theorem~\ref{main_theorem}.

\label{section:solution}
\end{section}

\begin{section}{Backwards error estimate}

In this section, we prove Theorem~\ref{overview:backwards}.

Although both $p$ and $p_b$ are defined on the real line,
we are only concerned with solutions of (\ref{introduction:original_equation})
on the interval $[0,1]$,  and so we only require estimates there.
Accordingly, throughout this section we  use $\|\cdot\|_\infty$ to denote
the $L^\infty\left([0,1]\right)$ norm.


We omit the proof of the following lemma, which is somewhat
long and technical but entirely elementary
(it can be established with the techniques found in ordinary differential
equation textbooks; see, for instance, Chapter 1 of \cite{Coddington-Levinson}).
\vskip 1em
\begin{lemma}
Suppose that $q:\mathbb{R} \to \mathbb{R}$ is infinitely differentiable and strictly
positive,
that $p(t)$ is defined by (\ref{inteq:definition_of_p}),
and that there exist real numbers $\eta_1 > 0$ and $\eta_2 > 0$ such that
\begin{equation}
\eta_1 \leq q(t) \leq \eta_2
\ \ \ \mbox{for all} \ \ 0 \leq t \leq 1,
\end{equation}
and
\begin{equation}
\left|p(t)\right|,\left|q'(t)\right| \leq \eta_2
\ \ \ \mbox{for all}\ \  0 \leq t \leq 1.
\end{equation}
Let
\begin{equation}
k = 20 \left(\frac{\eta_2}{\eta_1}\right)^2 + 8 \eta_2^2 + 10 \frac{\eta_2}{\eta_1} + 1
\end{equation}
and suppose also that $\epsilon >0$ is a real number such that
\begin{equation}
\epsilon < \frac{\eta_1}{2}.
\end{equation}
Suppose furthermore that $p_b:[0,1]\to\mathbb{R}$ is an infinitely differential
function such that
\begin{equation}
\|p-p_b\|_\infty \leq \epsilon \exp(-k).
\end{equation}
Then there exists an infinitely differentiable function $q_b:[0,1] \to \mathbb{R}$ 
such that 
\begin{equation}
 \frac{1}{q_b(t)} \left(
\frac{5}{4} \left(\frac{q_b'(t)}{q_b(t)}\right)^2
-\frac{q_b''(t)}{q_b(t)} \right) = p_b(t)
\ \ \ \mbox{for all} \ \ 0 \leq t \leq 1 
\end{equation}
and
\begin{equation}
\|q-q_b\|_\infty 
\leq \epsilon.
\end{equation}
\label{backwards:lemma1}
\end{lemma}
%

We derive a bound on the change in solutions of
the ordinary differential equation  (\ref{introduction:original_equation})
when  the coefficient $q(t)$ is perturbed.

%
%
%
\vskip 1em
\begin{lemma}
Suppose that $\lambda >0$, $\epsilon >0$, $\eta_1 >0$ and $\eta_2 > 0$ are real numbers.
Suppose also that $q:[0,1] \to \mathbb{R}$ is a continuously differentiable
function such that
\begin{equation}
\eta_1 \leq q(t) \leq \eta_2
\ \ \ \mbox{for all}\ \  0 \leq t \leq 1,
\end{equation}
that $p(t)$ is defined by the formula (\ref{inteq:definition_of_p}), and that
\begin{equation}
\left|p(t)\right| \leq \eta_2
\ \ \ \mbox{for all}\ \  0 \leq t \leq 1.
\end{equation}
Suppose furthermore that $q_b: [0,1] \to \mathbb{R}$ is a continuously differentiable
function such that
\begin{equation}
\left|q(t) - q_b(t)\right| \leq 
\frac{1}{2}
\frac{\eta_1 }{\lambda}
\exp\left(-\frac{\eta_2^{3/4}}{4}\right)
\epsilon
\ \ \ \mbox{for all}\ \  0 \leq t \leq 1.
\end{equation}
If $z(t)$ is a solution of the ordinary differential equation
\begin{equation}
z''(t) + \lambda^2 q(t) z(t) = 0
\ \ \ \mbox{for all}\ \  0 \leq t \leq 1
\end{equation}
and $z_0(t)$ is the unique solution of the ordinary differential equation
\begin{equation}
z_0''(t) + \lambda^2 q_b(t) z_0(t) = 0
\ \ \ \mbox{for all}\ \  0 \leq t \leq 1
\end{equation}
such that $z_0(0) = z(0)$ and $z_0'(0) = z'(0)$, then
\begin{equation}
\|z-z_0\|_\infty \leq \epsilon \|z\|_\infty.
\end{equation}
\label{backwards:lemma2}
\end{lemma}
\begin{proof}
We start by observing 
that the function $\psi(t) = z_0(t) - z(t)$ is 
the unique solution of the initial value problem
\begin{equation}
\left\{
\begin{aligned}
\psi''(t) + \lambda^2 q(t) \psi(t) &=  \lambda^2 f(t)
\ \ \ \mbox{for all}\ \  0 \leq t \leq 1\\
\psi(0) = \psi'(0) &= 0,
\end{aligned}
\right.
\label{backwards:lemma2:1}
\end{equation}
where
\begin{equation}
f(t) = 
 \left(q(t) -q _0(t)\right) \left(\psi(t) + z(t)\right).
\label{backwards:lemma2:2}
\end{equation}
We now apply the well-known Liouville-Green transformation 
to (\ref{backwards:lemma2:1})  in two steps.  
First, we introduce the function
\begin{equation}
\phi(t) = \left(q(t)\right)^{1/4} \psi(t),
\label{backwards:lemma2:2h}
\end{equation}
which is the solution of the initial value problem
\begin{equation}
\left\{
\begin{aligned}
\phi''(t) - \frac{q'(t)}{2q(t)} \phi'(t) + \lambda^2 q(t) \phi(t) 
&= \left(q(t)\right)^{1/4} \lambda^2 f(t) -\frac{1}{4} q(t) p(t)\phi(t)
\ \ \ \mbox{for all}\ \ 0 \leq t \leq 1\\
\phi(0) = \phi'(0) &= 0.  \\
\end{aligned}
\right.
\label{backwards:lemma2:3}
\end{equation}
%
%
Next we introduce the change of variables
\begin{equation}
x(t) = \int_0^t \sqrt{q(u)}\ du,
\label{backwards:lemma2:5}
\end{equation}
which transforms (\ref{backwards:lemma2:3}) into
\begin{equation}
\left\{
\begin{aligned}
\phi''(x) +  \lambda^2 \phi(x) &= \left(q(x)\right)^{-3/4} \lambda^2 f(x) - \frac{1}{4} p(x) \phi(x) 
\ \ \ \mbox{for all}\ \ 0 \leq x \leq x_1 \\
\phi(0) = \phi'(0) &= 0,
\end{aligned}
\right. 
\label{backwards:lemma2:6}
\end{equation}
where 
\begin{equation}
x_1 = \int_0^1 \sqrt{q(u)}\ du.
\label{backwards:lemma2:7}
\end{equation}
We now use the Green's function for the initial value problem (\ref{backwards:lemma2:6})
obtained from the Liouville-Green transform to conclude that for all $0 \leq x \leq x_1$, 
\begin{equation}
\phi(x) = 
\int_0^x \frac{\sin(\lambda(x-y))}{\lambda}
\left(\left(q(y)\right)^{-3/4} \lambda^2 f(y) - \frac{1}{4} p(y) \phi(y) \right)\ dy.
\label{backwards:lemma2:8}
\end{equation}
By inserting (\ref{backwards:lemma2:2}) into (\ref{backwards:lemma2:8}) 
and we obtain the inequality
\begin{equation}
\left|\phi(x)\right| \leq
\left(\|q-q_b\|_\infty\frac{\lambda}{\eta_1} + 
\frac{\eta_2}{4}
\right)
\int_0^x \left|\phi(y)\right|\ dy
+
\frac{\lambda}{\eta_1^{3/4}}\|q-q_b\|_\infty \|z\|_\infty
\ \ \ \mbox{for all}\ \ 0 \leq x \leq x_1.
\label{backwards:lemma2:9}
\end{equation}
We  apply Gronwall's inequality
(Theorem~\ref{preliminaries:theorem:gronwall} in Section~\ref{preliminaries:section:gronwall})
to (\ref{backwards:lemma2:9}) in order to conclude that
\begin{equation}
\left|\phi(x)\right| \leq
\frac{\lambda}{\eta_1^{3/4}} \|q-q_b\|_\infty 
\|z\|_\infty
\exp\left(
\left(\|q-q_b\|_\infty\frac{\lambda}{\eta_1} + 
\frac{\eta_2}{4}\right)x
\right)
\ \ \ \mbox{for all}\ \ 0 \leq x \leq x_1.
\label{backwards:lemma2:11}
\end{equation}
Combing (\ref{backwards:lemma2:2h}), (\ref{backwards:lemma2:11}) and the observation that 
\begin{equation}
x_1 = \int_0^1 \sqrt{q(u)}\ du \leq \eta_2^{1/2}
\label{backwards:lemma2:12}
\end{equation}
yields the inequality
\begin{equation}
\left|\phi(t)\right| 
\leq  
\frac{\lambda}{\eta_1}
\|q-q_b\|_\infty 
\exp\left(\frac{\lambda}{\eta_1} \|q-q_b\|_\infty \eta_2^{1/2}\right)
\exp\left(\frac{\eta_2^{3/4}}{4}\right)
\|z\|_\infty
\label{backwards:lemma2:13}
\end{equation}
for all $0 \leq t \leq 1$.  Now 
\begin{equation}
\frac{\lambda}{\eta_1}
\|q-q_b\|_\infty 
\exp\left(\frac{\lambda}{\eta_1} \|q-q_b\|_\infty \eta_2^{1/2}\right)
\exp\left(\frac{\eta_2^{3/4}}{4}\right)
< \epsilon
\label{backwards:lemma2:14}
\end{equation}
if and only if
\begin{equation}
\frac{\lambda}{\eta_1}
\eta_2^{1/2}\|q-q_b\|_\infty 
< W_0\left(\epsilon\eta_2^{1/2}
\exp\left(-\frac{\eta_2^{3/4}}{4}\right)\right),
\label{backwards:lemma2:15}
\end{equation}
where $W_0$ is the branch of the Lambert $W$ function which
is real-valued and greater than $-1$ on the interval $[-1/e,\infty)$
(see Section~\ref{section:preliminaries:LambertW}).
According to Theorem~\ref{preliminaries:LambertW:theorem3},
\begin{equation}
\frac{\lambda}{\eta_1}
\eta_2^{1/2}\|q-q_b\|_\infty 
< \frac{1}{2}\left(\epsilon\eta_2^{1/2}
\exp\left(-\frac{\eta_2^{3/4}}{4}\right)\right),
\label{backwards:lemma2:16}
\end{equation}
implies (\ref{backwards:lemma2:15}).  We algebraically simplify 
(\ref{backwards:lemma2:16}) in order to conclude that 
\begin{equation}
\|q-q_b\|_\infty 
< 
\frac{1}{2}
\frac{\eta_1 }{\lambda}
\exp\left(-\frac{\eta_2^{3/4}}{4}\right)
\epsilon
\label{backwards:lemma2:17}
\end{equation}
implies (\ref{backwards:lemma2:13}).
\end{proof}

By combining Lemmas~\ref{backwards:lemma1} and \ref{backwards:lemma2}
we obtain Theorem~\ref{overview:backwards}.

\label{section:backwards}
\end{section}

\begin{section}{Numerical experiments}

In this section, we describe numerical experiments which, {\it inter alia}, illustrate 
one of  the important consequences of the existence of nonoscillatory phase functions.  
Namely, that a large class of special functions can be evaluated to high accuracy
using a number of operations which does not grow with order.

Although the proof of Theorem~\ref{main_theorem} suggests a numerical
procedure for the construction of nonoscillatory phase functions, we utilize
a different procedure here.  It has the advantage that
the coefficient $q$ in the ordinary differential equation 
(\ref{introduction:original_equation}) need not be extended outside of the interval
on which the nonoscillatory phase function is constructed.  A paper describing
this work is in preparation.

The code we used for these calculations was written in Fortran and
compiled with the Intel Fortran Compiler version 12.1.3.
All calculations were carried out in double precision arithmetic
on a desktop computer equipped with an Intel Xeon X5690 CPU
running at 3.47 GHz.

\begin{subsection}{A nonoscillatory solution of the logarithm form of Kummer's equation.}
In this experiment, we illustrate Theorem~\ref{main_theorem}
in Section~\ref{section:overview}.  We first construct a nonoscillatory solution $r$ of the logarithm form of
Kummer's equation
\begin{equation}
r''(t) -\frac{1}{4}\left(r'(t)\right)^2 + 4\lambda^2 \left(\exp(r(t))-q(t)\right) = 0
\end{equation}
on the interval $[-1,1]$,  where $\lambda$ = 1,000 and $q$  
is the function $[-1,1]\to\mathbb{R}$ defined by the formula
\begin{equation}
q(t) = \left(3 + \frac{1}{1+10t^2} + t^3\cos(5t) \right).
\label{numerics:q}
\end{equation}
Then we compute the 500 leading Chebyshev coefficients of 
$q$ and $r$.

We display the results of this experiment in Figures~\ref{numerics:simple:figure1}
and \ref{numerics:simple:figure2}.
Figure~\ref{numerics:simple:figure1} 
contains plots of the functions $q$ and $r$, while Figure~\ref{numerics:simple:figure2}
contains a plot of the base-$10$ logarithms of the
absolute values of the leading Chebyshev coefficients of $q$ and $r$.   

We observe that, consistent with Theorem~\ref{main_theorem},
the Chebyshev coefficients of both $r$ and $q$ decay exponentially,
although those of $r$ decay at a slightly slower rate.
\label{section:numerics:simple}

\end{subsection}

\begin{subsection}{Evaluation of Legendre polynomials.}
In this experiment, we compare the cost of evaluating Legendre polynomials
of large order using the standard recurrence relation with the cost
of doing so with a nonoscillatory phase function.

For any integer $n \geq 0$, the Legendre polynomial $P_n(x)$ of order $n$ is a solution
of the second order differential equation
\begin{equation}
(1-t^2) y''(t) - 2t y'(t) + n(n+1) y(t) = 0.
\label{numerics:legendre:1}
\end{equation}
Equation (\ref{numerics:legendre:1}) can be put into the standard form
\begin{equation}
\psi''(t) + \left(\frac{1+n-nt^2-n^2(t^2-1)}{(1-t^2)^2}\right)  \psi(t) = 0
\label{numerics:legendre:2}
\end{equation}
by introducing the transformation
\begin{equation}
\psi(t) = \sqrt{1-t^2}\ y(t).
\label{numerics:legendre:3}
\end{equation}
Legendre polynomials satisfy the well-known three term recurrence relation
\begin{equation}
(n+1) P_{n+1}(t) = (2n+1) t P_n(t) - n P_{n-1}(t).
\label{numerics:legendre:5}
\end{equation}
See, for instance, \cite{NISTHandbook} for a discussion of the these
and other properties of Legendre polynomials.

For each of $9$ values of $n$, we proceed as follows.
We  sample $1000$ random points 
\begin{equation}
t_1,t_2,\ldots,t_{1000}
\end{equation}
from the uniform distribution on the interval $(-1,1)$.  
Then we evaluate the Legendre polynomial
of order $n$ using the recurrence relation (\ref{numerics:legendre:5})
at each of the points $t_1,t_2,\ldots,t_{1000}$.  Next, we construct a nonoscillatory phase
function for the ordinary differential equation
(\ref{numerics:legendre:2})  and use it evaluate the Legendre polynomial
of order $n$ at each of the points $t_1,t_2,\ldots,t_{1000}$.  Finally, for each
integer $j=1,\ldots,1000$, we compute the error in the approximation
of $P_n(t_j)$ obtained from the nonoscillatory phase function by comparing it to
the value obtained using the recurrence relation (we regard the recurrence relation
as giving the more accurate approximation).
 
The results of this experiment are shown in Table~\ref{numerics:legendre:table1}.
There, each row correponds to value of $n$.
That value is listed $n$, as is the  time required to compute each phase function for
that value of $n$,  the average time required to
evaluate the Legendre polynomial of order $n$ using the recurrence relation,
the average cost of evaluating the Legendre polynomial of order $n$
with the nonoscillatory phase function,
and the largest of the absolute errors in the approximations of
the quantities  $$P_n(t_1), P_n(t_2), \ldots,P_n(t_{1000})$$ 
obtained via the phase function method.

This experiment reveals that, as expected, 
the cost of evaluating $P_n(t)$ using the recurrence relation
(\ref{numerics:legendre:5})  grows as $O(n)$ while the cost of
doing so with  nonoscillatory phase function is independet of
order.

However, it also exposes a limitation of phase functions.
The values of $P_n(t)$ are obtained in part by evaluating sine and cosine of a 
phase function whose magnitude is on the order of $n$.  This imposes limitations on 
the accuracy of the method due to the well-known difficulties in  evaluating periodic 
functions of large arguments.

Figure~\ref{numerics:legendre:figure1} contains a plot of the nonoscillatory phase
function for the equation (\ref{numerics:legendre:2}) when $n=$1,000,000.

\label{section:numerics:legendre}
\end{subsection}


\begin{subsection}{Evaluation of Bessel functions.}

In this experiment, we compare the cost of evaluating Bessel functions
of integer order via the standard recurrence relation with that
of doing so using a nonoscillatory phase function.

We will denote by $J_\nu$ the Bessel function of the first kind
of order $\nu$.  It is a solution of the second order
differential equation
\begin{equation}
t^2 y''(t) + t y'(t)  + (t^2-\nu^2) y(t) = 0,
\label{numerics:bessel:1}
\end{equation}
which can be brought into the standard form
\begin{equation}
\psi''(t) + \left(1-\frac{\lambda^2-1/4}{t^2} \right)\psi(t) = 0
\label{numerics:bessel:2}
\end{equation}
via the transformation
\begin{equation}
\psi(t) = \sqrt{t}\ y(t).
\label{numerics:bessel:3}
\end{equation}
An inspection of (\ref{numerics:bessel:2})  reveals that
$J_\nu$ is nonoscillatory on the interval
\begin{equation}
\left(0,\frac{1}{2}\sqrt{4\nu^2-1}\right)
\end{equation}
and oscillatory on the interval
\begin{equation}
\left(\frac{1}{2}\sqrt{4\nu^2-1},\infty\right).
\label{numerics:bessel:4}
\end{equation}
The Bessel functions satisfy the three-term recurrence relation
\begin{equation}
J_{\nu+1}(t) = \frac{2\nu}{t} J_\nu(t) - J_{\nu-1}(t).
\label{numerics:bessel:5}
\end{equation}
The recurrence (\ref{numerics:bessel:5}) is numerically unstable in the forward
direction; however, when evaluated in the direction of decreasing
index, it yields a stable mechanism for evaluating Bessel functions
of integer order (see, for instance, Chapter 3 of \cite{NISTHandbook}).

For each of $9$ values of $n$, we proceed as follows.
First, we  sample $1000$ random points 
\begin{equation}
t_1,t_2,\ldots,t_{1000}
\end{equation}
from the uniform distribution on the interval $[2n,3n]$.  We then use
the recurrence relation (\ref{numerics:bessel:5}) to evaluate the Bessel function
$J_n$ of order $n$ at the points $t_1,t_2,\ldots,t_{1000}$.  Next, we construct a  nonoscillatory
phase function for the equation (\ref{numerics:bessel:3}) on the interval
$[2n, 3n]$ and use it to evaluate $J_n$ at the points
$t_1,t_2,\ldots,t_{1000}$.  
Finally, for each integer $j=1,\ldots,1000$,
 we compute the error in the approximation of $J_n(t_j)$ obtained 
from the nonoscillatory phase function by comparing it to
the value obtained using the recurrence relation (once again
we regard the recurrence relation
as giving the more accurate approximation).

The results of this experiment are displayed in 
Table~\ref{numerics:bessel:table1}.  There, each row corresponds
to one value of $n$.   In addition to that value of $n$, it lists
the time required to compute the phase function at order $n$,
 the average cost of evaluating $J_n$ using the recurrence relation,
the average cost of evaluating it with the nonoscillatory phase function,
and the largest of the absolute errors in the approximations of
the quantities 
$$J_n(t_1), J_n(t_2), \ldots, J_n(t_{1000})$$ 
obtained via the phase function method.

We observe that while the cost of evaluating $J_n$ using the recurrence
relation (\ref{numerics:bessel:5}) grows as $O(n)$, the time
taken by the nonoscillatory phase function approach scales
as $O(1)$.  We also note that, as in the case of Legendre polynomials,
there is some loss of accuracy with the phase function
method due to the difficulties of evaluating trigonometric
functions of large arguments.

\label{section:numerics:bessel}
\end{subsection}

\label{section:numerics}
\end{section}

\begin{section}{Conclusions}

We have shown that the solutions of a large class of second order differential equations
can be accurately represented using nonoscillatory phase functions.

We have also presented the results of numerical experiments which demonstrate
one of the applications of nonoscillatory phase functions:
the  evaluation of special functions at a cost which is  independent of order.
An efficient algorithm for the evaluation of highly oscillatory special functions 
will be reported at a later date.

A number of open issues and questions related to this work remain.
Most obviously, a further investigation of  the integral equation 
(\ref{inteq:integral_equation}) and the conditions under which it admits an exact
solution is warranted.  Moreover, there are applications of nonoscillatory
phase functions beyond the evaluation of special functions which should be explored.
And, of course, the generalization of these results to higher dimensions is 
of great interest. The authors are vigorously pursuing these avenues of research.

\label{section:conclusion}
\end{section}

\begin{section}{Acknowledgements}
Zhu Heitman was supported in part by the Office of Naval Research
under contracts ONR N00014-10-1-0570 and ONR N00014-11-1-0718.
James Bremer was supported in part by a fellowship from the Alfred P. Sloan
Foundation.  Vladimir Rokhlin was supported in part by 
Office of Naval Research contracts
ONR N00014-10-1-0570 and ONR N00014-11-1-0718, and by the Air Force Office of Scientific
Research under contract AFOSR FA9550-09-1-0241.
\end{section}

\begin{section}{References}
\bibliographystyle{acm}
\bibliography{kummer}
\end{section}

\vfill\eject

\begin{figure}[h!!]
\begin{center}
\includegraphics[width=.8\textwidth]{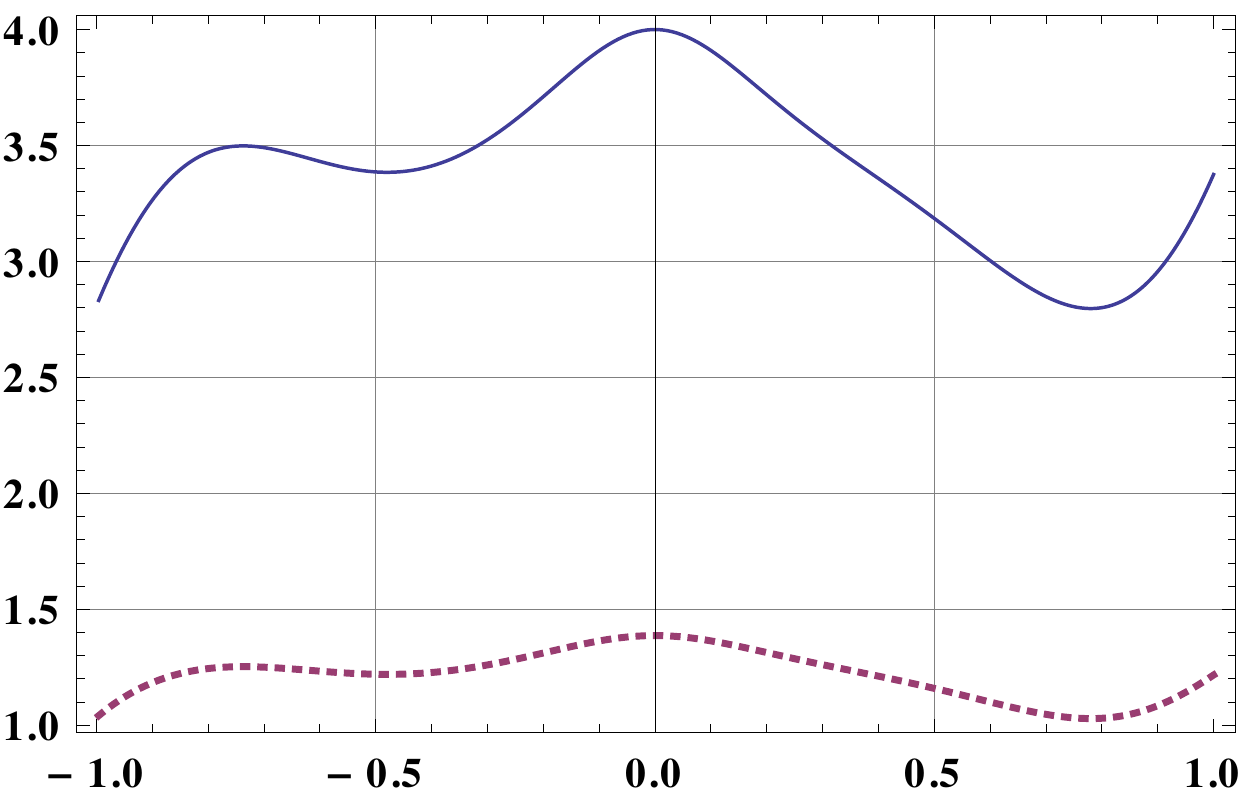}
\caption{The function $q$ 
defined by formula (\ref{numerics:q})
in Section~\ref{section:numerics:simple} (solid line)
and the corresponding solution $r$ of the logarithm form of Kummer's equation
(\ref{inteq:kummer_logarithm_form})
when $\lambda = $ 1,000 (dotted line).
}
\label{numerics:simple:figure1}
\end{center}
\end{figure}
\vfill

\begin{figure}[h!!]
\begin{center}
\includegraphics[width=.8\textwidth]{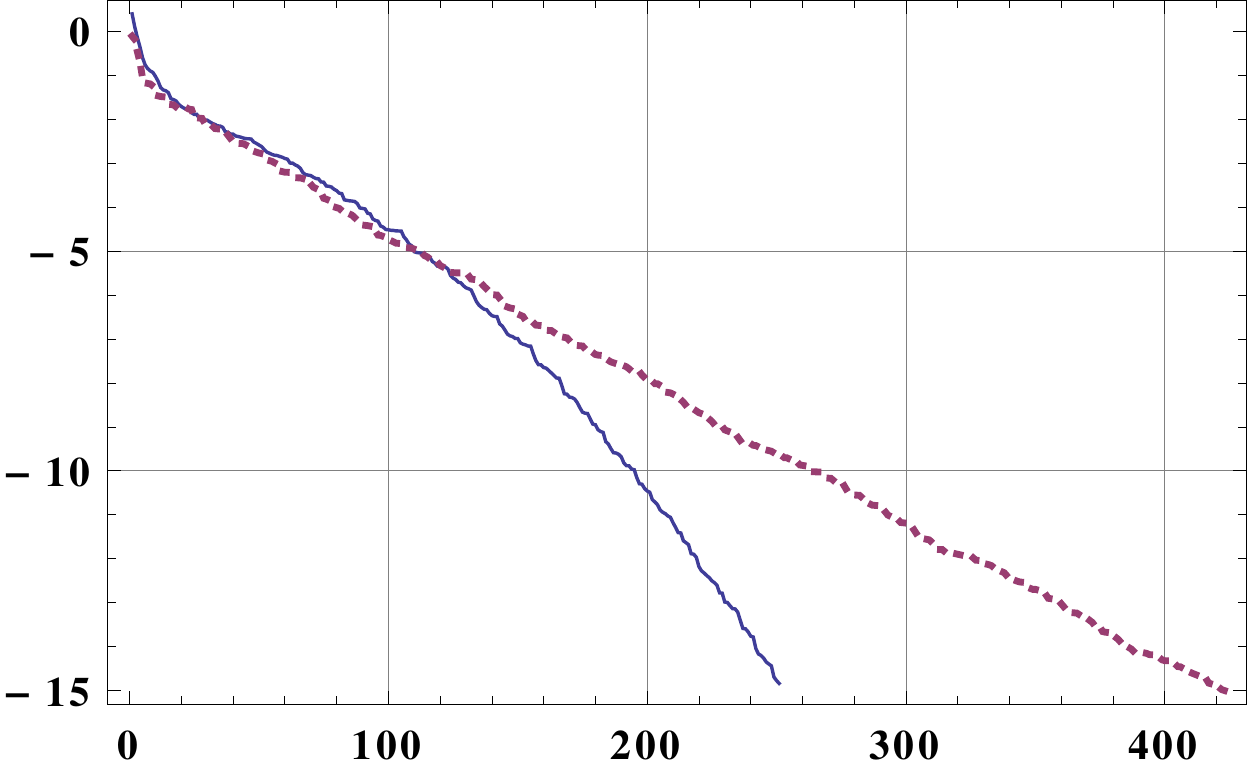}
\caption{The base-$10$ logarithms of the leading Chebyshev coefficients
of the function $q$ 
defined by formula (\ref{numerics:q})
in Section~\ref{section:numerics:simple} (solid line)
and of the associated nonoscillatory
solution $r$ of equation the logarithm form of Kummer's equation 
(\ref{inteq:kummer_logarithm_form})
when $\lambda$ = 1,000
(dotted line).
}
\label{numerics:simple:figure2}
\end{center}
\end{figure}

\vfil\eject
\begin{table}[h!!]
\begin{center}
\begin{tabular}{ccccc}
\multirow{2}{*}{$n$}    &   Phase function & Avg. phase function & Avg. recurrence& Largest \\
& construction time & evaluation time & evaluation time & absolute error\\
\midrule
\addlinespace[.5em]
$10^1$    & 1.55\e{-1} secs     & 1.29\e{-6} secs & 5.82\e{-8} secs & 5.16\e{-14} \\
\addlinespace[.25em]
$10^2$    & 1.76\e{-1} secs     & 1.29\e{-6} secs & 9.73\e{-7} secs & 1.59\e{-13} \\
\addlinespace[.25em]
$10^3$    & 1.57\e{-1} secs     & 1.29\e{-6} secs & 1.03\e{-5} secs & 6.13\e{-13} \\
\addlinespace[.25em]
$10^4$    & 1.55\e{-1} secs     & 1.29\e{-6} secs & 1.04\e{-4} secs & 1.20\e{-12} \\
\addlinespace[.25em]
$10^5$    & 1.56\e{-1} secs     & 1.31\e{-6} secs & 1.04\e{-3} secs & 9.79\e{-12} \\
\addlinespace[.25em]
$10^6$    & 1.58\e{-1} secs     & 1.40\e{-6} secs & 9.81\e{-3} secs & 2.40\e{-11} \\
\addlinespace[.25em]
$10^7$    & 1.65\e{-1} secs     & 1.40\e{-6} secs & 9.69\e{-2} secs & 8.59\e{-11} \\
\addlinespace[.25em]
$10^8$    & 1.87\e{-1} secs     & 1.42\e{-6} secs & 9.68\e{-1} secs & 1.71\e{-10} \\
\addlinespace[.25em]
$10^9$    & 2.05\e{-1} secs    & 1.34\e{-6} secs & 9.68\e{-0} secs & 6.11\e{-10} \\
\end{tabular}
\caption{
{\bf The evaluation of Legendre polynomials}.
A comparison of the time required to evaluate the Legendre polynomial of order
$n$ using the standard recurrence relation and the time necessary to evaluate
it using a nonoscillatory phase function.
The recurrence relation approach scales as $O(n)$ while
the phase function approach scales as $O(1)$.
}
\label{numerics:legendre:table1}
\end{center}
\end{table}

\vfil

\begin{figure}[h!!]
\begin{center}
\includegraphics[width=.8\textwidth]{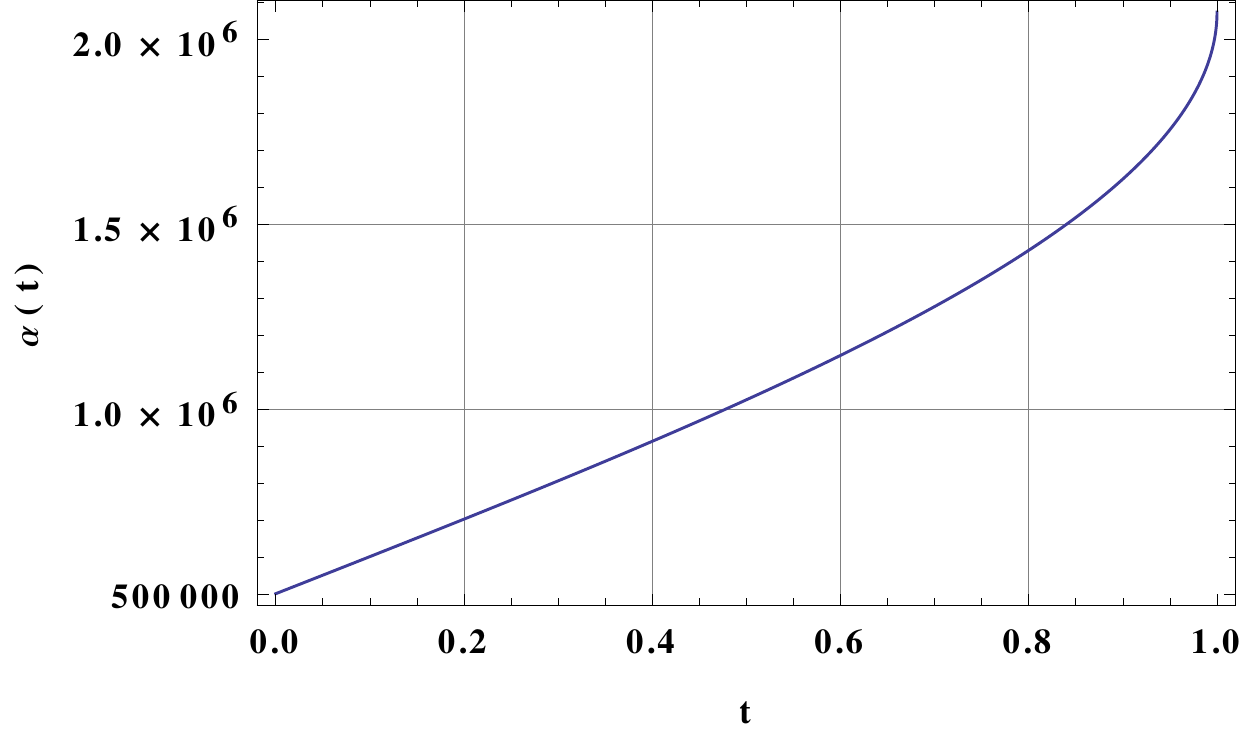}
\caption{
{\bf A phase function for Legendre's differential equation}.
A plot of the nonoscillatory phase function associated with Legendre's equation
(\ref{numerics:legendre:1})
at order  $n =$ 1,000,000.
It is sufficient to construct the phase function on the interval $[0,1)$
due to the symmetry properties of Legendre's differential equation.
}
\label{numerics:legendre:figure1}
\end{center}
\end{figure}

\vfil\eject

\vfil\eject
\begin{table}[h!!]
\begin{center}
\begin{tabular}{ccccc}
\multirow{2}{*}{$n$}
&   Phase function  & Avg. phase function & Avg. recurrence  & Largest \\
& construction time & evaluation time     & evaluation time  & absolute error\\
\midrule
\addlinespace[.5em]
$10^1$    & 5.23\e{-1} secs     & 1.30\e{-6} secs & 1.99\e{-6} secs & 2.81\e{-14} \\
\addlinespace[.25em]
$10^2$    & 5.39\e{-1} secs     & 1.31\e{-6} secs & 7.29\e{-6} secs & 7.85\e{-14} \\
\addlinespace[.25em]
$10^3$    & 5.36\e{-1} secs     & 1.37\e{-6} secs & 4.87\e{-5} secs & 2.40\e{-13} \\
\addlinespace[.25em]
$10^4$    & 5.52\e{-1} secs     & 1.33\e{-6} secs & 4.35\e{-4} secs & 1.01\e{-12} \\
\addlinespace[.25em]
$10^5$    & 5.46\e{-1} secs     & 1.49\e{-6} secs & 4.11\e{-3} secs & 3.18\e{-12} \\
\addlinespace[.25em]
$10^6$    & 5.81\e{-1} secs     & 1.44\e{-6} secs & 4.24\e{-2} secs & 8.57\e{-12} \\
\addlinespace[.25em]
$10^7$    & 6.41\e{-1} secs     & 1.45\e{-6} secs & 4.36\e{-1} secs & 5.98\e{-11} \\
\addlinespace[.25em]
$10^8$    & 7.00\e{-1} secs     & 1.35\e{-6} secs & 4.39\e{+0} secs & 1.14\e{-10} \\
\addlinespace[.25em]
$10^9$    & 1.26\e{+0} secs     & 1.41\e{-6} secs & 4.42\e{+1} secs & 2.43\e{-10} \\
\end{tabular}
\caption{
{\bf The evaluation of Bessel functions}.
A comparison of the time required to evaluate the Bessel function $J_n$
using the standard recurrence relation 
with that required to evaluate it  using a nonoscillatory phase function.
All of the points at which $J_n$ was evaluated were in the interval $[2n,3n]$. 
The recurrence relation approach scales as $O(n)$ in the order $n$ 
while the time required by the phase function method is $O(1)$.
}
\label{numerics:bessel:table1}
\end{center}
\end{table}

\vfill\eject

\end{document}